\documentclass[12pt,reqno]{amsart}

\usepackage{latexsym}
\usepackage{amscd, amsfonts, mathrsfs, amsmath, amssymb, amsthm, mathtools}
\usepackage{baskervald}
\usepackage[margin=1in]{geometry}
\usepackage{tikz-cd}
\usepackage{verbatim}
\usepackage{graphicx}
\usepackage{mathrsfs,hyperref}
\usepackage{lmodern}
\usepackage[T1]{fontenc}
\usepackage{thmtools,thm-restate}
\newtheorem{theorem}{Theorem}[section]
\newtheorem{lemma}[theorem]{Lemma}

\theoremstyle{definition}
\newtheorem{definition}[theorem]{Definition}
\newtheorem{cor}[theorem]{Corollary}
\newtheorem{example}[theorem]{Example}

\theoremstyle{remark}
\newtheorem{remark}[theorem]{Remark}

\numberwithin{equation}{section}

\newcommand{\mP}{{\mathbb P}}
\newcommand{\mA}{{\mathbb A}}
\newcommand{\C}{{\mathbb C}}
\newcommand{\Z}{{\mathbb Z}}
\DeclareMathOperator{\Br}{Br}
\DeclareMathOperator{\Res}{Res}
\DeclareMathOperator{\Frac}{Frac}
\DeclareMathOperator{\CH}{CH}
\DeclareMathOperator{\Pic}{Pic}
\DeclareMathOperator{\coker}{coker}
\DeclareMathOperator{\Gal}{Gal}
\DeclareMathOperator{\Spec}{Spec}
\DeclareMathOperator{\Proj}{Proj}

\begin{document}

\title{Rationality of Brauer-Severi surface bundles over rational 3-folds}

\author{Shitan Xu}
\address{ Department of Mathematics, Michigan State University,
 East Lansing, MI 48824-1027, USA}

\email{xushita1@msu.edu}


\date{\today}


\begin{abstract}
We give a sufficient condition for a Brauer-Severi surface bundle over a rational 3-fold to not be stably rational. Additionally, we present an example that satisfies this condition and demonstrate the existence of families of Brauer-Severi surface bundles whose general members are smooth and not stably rational.
\end{abstract}

\maketitle

\section{Introduction}
This paper is motivated by the study of the stable rationality of conic bundles and Brauer-Severi surface bundles. Let $X$ be a projective variety over an algebraically closed field $k$. We say that $X$ is {\it rational} if $X$ is birational to a projective space $\mP^n_{k}$ for some natural number $n$. We call $X$ {\it stably rational} if there exists a natural number $m$ such that $X \times_{k} \mP^m_{k}$ is rational. In \cite{MR3849287}, an example of a quadratic surface bundle over $\mP^2$ that is not stably rational and has a nontrivial unramified Brauer group is constructed. This example is realized as a divisor in $\mP^2 \times \mP^3$ of bidegree (2,2), with the quadratic surface bundle structure induced by the first projection.

In \cite{MR4069651}, this elegant example was examined from a different perspective: it naturally becomes a conic bundle over $\mP^3$ via the second projection. This structure allowed the authors to establish a sufficient condition \cite[Thm.~2.6]{MR4069651} for a conic bundle over $\mP^3$ to not be stably rational. Furthermore, they introduced a new example of a flat family of conic bundles over $\mP^3$, where a very general member is not stably rational, using a theorem of Voisin \cite[Thm.~2.1]{MR3359052}, in the form of \cite[Thm.~2.3]{MR3481353}. Voisin’s theorem serves as a powerful tool in the study of the stable rationality of families of varieties. This approach is commonly referred to in the literature as the specialization method. In a recent paper \cite{pirutka2023cubicsurfacebundlesbrauer}, Pirutka introduced the notion of the relative unramified cohomology group, which combines the approach of constructing nontrivial unramified Brauer class via fibrations (\cite{10.1112/plms/s3-25.1.75},\cite{Colliot-Thelene1989}) and the method given in \cite{MR3909896} and \cite{MR4013741} to avoid geometric construction of universally $CH_0$-trivial desingularization. Our proof in section 6 applied this idea in a different manner.

A natural next step following these advancements is to investigate the stable rationality of Brauer-Severi surface bundles over $\mP^3$. In 2017, Kresch and Tschinkel introduced good models of Brauer-Severi surface bundles using the concept of a root stack \cite{MR3993007}. With this definition, they constructed a flat family of Brauer-Severi surface bundles over $\mP^2$ \cite[Thm.~1]{MR4046975}, in which the general member is smooth and not stably rational.

This paper begins by generalizing Theorem 2.6 of \cite{MR4069651} to Brauer-Severi surface bundles over 3-folds. After constructing a new (singular) example with a nontrivial unramified Brauer group, we obtain the following result:
\begin{restatable}{theorem}{primetheorem}
\label{t:final}
There exists a flat projective family of Brauer-Severi surface bundles over $\mP^3_{\C}$, where a general fiber in this family is smooth and not stably rational. 
\end{restatable}
The structure of this paper is as follows: In Section \ref{s:2}, we review some basic facts about Brauer groups and introduce the unramified Brauer group, a stably birational invariant. By definition, this is the subgroup of the Brauer group of the function field, whose elements arise from the Brauer classes of a smooth model. Recently, significant progress has been made using the unramified Brauer group and the specialization method to show that a very general member of certain classes of varieties is not stably rational (\cite[Section~2.1]{MR3821181}).

In Section \ref{s:3}, we generalize a criterion for the stable rationality of conic bundles over 3-folds (\cite[Thm.~2.6]{MR4069651}) to the case of Brauer-Severi surface bundles. This provides a tool to construct an explicit (singular) Brauer-Severi surface bundle over $\mP^3_{\C}$ with a nontrivial unramified Brauer group, which we explore in Section \ref{s:4}. In Section \ref{s:5}, we verify that this example is indeed a Brauer-Severi surface bundle.

In Section \ref{s:6}, we show that our example \ref{e} satisfies the hypotheses required by the specialization method introduced by Voisin in 2014 \cite{MR3359052}, and further developed by Colliot-Thélène and Pirutka in 2016 \cite{MR3481353}. Following Schreieder’s approach \cite[Proposition~26]{MR3909896}, we verify this using a purely cohomological criterion. Finally, in Section \ref{s:7}, we prove Theorem \ref{t:final} by constructing a flat family of Brauer-Severi surface bundles over $\mP^3_{\C}$, where the general member is smooth and includes Example \ref{e} as a member. Detailed calculations are provided separately in Appendix \ref{s:A}.
\subsection{Acknowledgements} I thank my advisor, Professor Rajesh Kulkarni, for his valuable discussions and suggestions on this work. I also extend my gratitude to Prof.~Pirutka, Prof.~Auel, Prof.Kresch and Prof. Schreieder for their comments on an earlier draft. During this project, I was partially supported by NSF grant DMS-2101761.
\section{Background}
\label{s:2}
\subsection{Brauer groups and purity}

For the definition and a detailed treatment of Brauer groups of fields and schemes, see \cite{MR2266528} or \cite{MR4304038}. We use standard conventions in Galois cohomology below. In particular, $$H^i(L, M) = H^i(\textup{Gal}(\bar{L}/L), M).$$

Let $L$ be the function field of an integral scheme $Z$ over the field of complex numbers, $\C$. The Kummer sequence provides a Galois cohomology sequence that identifies the 3-torsion subgroup of $\Br(L)$ with the second Galois cohomology group of $L$ with constant coefficients $\Z/3$:
$$H^2(L,\Z/3)\cong \Br(L)[3].$$

By the Merkurjev-Suslin theorem \cite[Thm.~2.5.7]{MR2266528}, $\Br(L)[3]$ is generated by the equivalence classes of cyclic algebras of order 3. These are denoted by $(a,b)_{\omega}$ and are defined by
$$(a,b)_{\omega}= \left<x,y|x^3=a,y^3=b,xy=\omega yx \right>,$$
where $a,b\in L^{\times}$ and $\omega$ is a primitive third root of unity. 

Now let $\nu$ be a discrete valuation of $L$ with residue field $k(\nu)$, we have the following residue maps \cite[Section~6.8]{MR2266528}:
$$ \partial_{\nu}^{1}: H^1(L,\Z/3)\to H^0(k(\nu),\Z/3)$$ 
$$ \partial_{\nu}^{2}: H^2(L,\Z/3)\to H^1(k(\nu),\Z/3)$$ 
By Kummer theory, we can identify these residue maps as:
$$\partial_{\nu}^{1}: L^{\times}/L^{\times 3}\to \Z/3$$
$$\partial_{\nu}^{2}: \Br(L)[3]\to k(\nu)^{\times}/k(\nu)^{\times 3}$$

\begin{lemma}
\label{l:residue}
With notations as above, the two residue maps are defined by:
$$\partial_{\nu}^{1}([a]) =  \nu(a)(\text{mod} \  3)$$
$$\partial_{\nu}^{2}([(a,b)_{\omega}]) =  (-1)^{\nu(a)\nu(b)}{\frac{a^{\nu(b)}}{b^{\nu(a)}}}\mod  k(\nu)^{\times 3} $$
\end{lemma}
\begin{proof}
See \cite[Thm.~2.18]{MR3616009}
\end{proof}

\begin{definition}
\label{d:unramified}
    The 3-torsion of unramified Brauer group of a field $L$ over another field $k$, denoted by $H^2_{nr}(L/k,\Z/3\Z)$ or $\Br_{nr}(L/k)[3]$, is the intersection of kernels of all residue maps $\partial_{\nu}^2$, where $\nu$ take values in all divisorial valuations of $L$ which are trivial on $k$. 
\end{definition}

From the discussions in \cite{MR1327280} and \cite[Corollary~6.2.10]{MR4304038}, the unramified Brauer group of L is a stably birational invariant for any model X of L, whether the model is nonsingular or singular. Here, a model X of L refers to an integral projective variety with function field L. If the model is nonsingular, we have the following:

\begin{lemma}
Let $k$ be a field with characteristic not 3. Let $X$ be a regular, proper, integral variety over $k$ with function field $L$, then we have
$$\Br(X)[3]\cong \Br_{nr}(L/k)[3]$$
\end{lemma}
\begin{proof}
This is a direct result of \cite[Prop.~3.7.8]{MR4304038}   
\end{proof}

Given a 3-torsion Brauer class in the Brauer group of the function field L, it is not easy to determine whether it belongs to the unramified subgroup using the definition alone, as there are usually too many divisorial valuations to consider. In \cite{MR1327280}, several theorems are established that reduce the number of valuations needed to check whether a Brauer class is unramified. Specifically, it suffices to check valuations corresponding to prime divisors in a smooth model. This result follows from the {\it purity} property of unramified Brauer groups. For more details on these theorems, see \cite{MR1327280} and \cite[Section3.7]{MR4304038}. A similar discussion can be found in \cite[Section2]{MR4069651}.

\subsection{Brauer-Severi surface bundles}

We use the definition of Brauer-Severi surface bundles from \cite[Def.~4.1]{MR3993007}:
\begin{definition}
\label{d:bundle}
Let $B$ be a locally Noetherian scheme, in which 3 is invertible in the local rings. A {\it Brauer-Severi surface bundle} over $B$ is a flat projective morphism $\pi: Y \to B$ such that the fiber over every geometric point of $B$ is isomorphic to one of the following:
\begin{itemize}
    \item $\mP^2$
    \item The union of three standard Hirzebruch surfaces $\mathbb{F}_1$, meeting transversally, such that any pair of them meets along a fiber of one and the $(-1)$-curve of the other.
    \item An irreducible scheme whose underlying reduced subscheme is isomorphic to the cone over a twisted cubic curve. 
\end{itemize}
\end{definition}

 Let $\pi: Y\to B$ be a Brauer-Severi surface bundle over a smooth projective rational threefold $B$ over a field $k$ whose generic fiber is smooth, and let $S=\{b\in B|\pi^{-1}(b) \  \text{is singular}\}$ denote its discriminant locus. We consider $S$ with its reduced closed subscheme structure in $B$, and since $B$ is Noetherian, $S$ consists of finitely many irreducible components, say $S_1,\cdots,S_n$.

In the case of conic bundles considered in \cite[Def.~2.4]{MR4069651}, the authors focused on a special kind of conic bundles with a \textit{good} discriminant locus. We will generalize the definition of a \textit{good} discriminant locus to Brauer-Severi surface bundles in this context:

\begin{definition}
\label{d:good}
    We say the discriminant locus $S$ is {\it good} if the following conditions are satisfied:
    \begin{enumerate}
   
    \item Each irreducible component of $S$ is reduced. (This is assumed above.)
    \item The fiber $Y_s$ over a general $s\in S_i$ for each irreducible component $S_i$ is geometrically the union of three standard Hirzebruch surfaces $\mathbb{F}_1$ described in Definition \ref{d:bundle}.
    \item The natural triple cover of $S_i$ induced by  $\pi: Y\to B$ is irreducible.
    \item By $(3)$, the fiber $Y_{F_{S_i}}$ over the generic point of $S_i$ is irreducible. Thus, there is a natural map $\tau$ from the cubic classes of function field of $S_i$ to the cubic classes of function field of $Y_{F_{S_i}}$. We assume the cubic extension over the function field of $S_i$ induced by $\pi: Y\to B$ generates the kernel of $\tau$.
    
    \end{enumerate}
\end{definition}
\begin{remark}
The last requirement is automatically satisfied in the case of conic bundles with the suitable definition given in \cite[Thm.~2.6]{MR4069651}. However, in our case, this is not generally true. Note that the example we provided meets all these requirements (See Example \ref{e}).
\end{remark}
\begin{remark}
By Lemma \ref{l:residue}, those $S_i$ are precisely those irreducible surfaces such that the Brauer class of the generic fiber have a nontrivial residue along $S_i$. In particular, the discriminant locus is a divisor of the base with pure-dimensional irreducible components.
\end{remark}

We end this section with a lemma that generalizes \cite[Lemma~2.3]{MR4069651} to the 3-torsion case:
\begin{lemma}
\label{l:ses}
Let $S$ be a smooth nonsplit Brauer-Severi surface over an arbitrary field $K$ (in particular, $S_{\bar{K}}\cong \mP^2_{\bar{K}}$). Then the pullback map $\Br(K)\to \Br(S)$ induces an exact sequence:
$$0\to \Z/3\to \Br(K)\to \Br(S)\to 0,$$
where the kernel is generated by the Brauer class $\alpha\in \Br(K)[3]$ determined by $S$.
Furthermore, if characteristic of $K$ $\neq 3$ and $K$ contains a primitive $9$-th root of unit, then the above exact sequence restricts to :
$$0\to \Z/3\to \Br(K)[3]\to \Br(S)[3]\to \Z/3\to 0$$
\end{lemma}

\begin{proof}
We first prove the exactness of the first sequence. Recall that the kernel $\textup{ker}(\Br(K)\to \Br(S))$ is given by Amitsur's theorem \cite[Thm.~5.4.1]{MR2266528}. To show surjectivity of $\Br(K)\to \Br(S)$, consider the separable closure $K^s$ of $K$, and let $\Gamma=\textup{Gal}(K^s/K)$. Then we have the Hochschild-Serre spectral sequence 
$$H^p(\Gamma,H^q(S_{K^s},\mathbb{G}_m))\Rightarrow H^{p+q}(S,\mathbb{G}_m).$$
The low degree exact sequence reads 
$$0\to \Pic(S)\to \Pic(S_{K^s})^{\Gamma}\to \Br(K)\to \ker\left(\Br(S)\to \Br(S_{K^s})^{\Gamma}\right)\to H^1(\Gamma,\Z)$$
By the definition of a Brauer-Severi Variety, we know 
$S_{K^s}\cong \mP^2_{K^s}$.
Hence we have $\Br(S_{K^s})\cong \Br(K^s)=0$. Since we clearly have $H^1(\Gamma,\Z)=0$, it follows that 
$\Br(K)\to \Br(S)$ is surjective.

For the second part, consider any $a\in \Br(S)[3]$. There exists a lift $a'\in \Br(K)$ of $a$ such that $3a'\mapsto 0\in \Br(S)$. Therefore, $3a'\in \ker\left(\Br(K)\to \Br(S)\right)\cong \Z/3$.
In other words, $9a'=0\in \Br(K)$. Hence, $\Br(K)[9]$ surjects onto $\Br(S)[3]$. Notice that by the description of the kernel in Amitsur's theorem, we clearly have 
$$\ker\left(\Br(K)[3]\to \Br(S)[3]\right) \cong \Z/3.$$
To compute the cokernel, consider the short exact sequence of trivial $\Gamma$-modules:
$$1\to \mu_3\to \mu_9\to \mu_3\to 1,$$
and its associated long exact sequence:
$$\cdots\to H^1(K,\mu_3)\xrightarrow[]{\partial^1}H^2(K,\mu_3)\to H^2(K,\mu_9)\to H^2(K,\mu_3)\xrightarrow[]{\partial^2}H^3(K,\mu_3)\to \cdots$$
We claim the boundary maps $\partial^1,\partial^2$ in the above exact sequence are zero. Indeed, let $\omega$ be a primitive third root of unit. By the proof of \cite[Lemma 4.1]{44267b53-f4ad-340a-8e65-df1c74888743}, $\partial^1$ is given by the cup product with $[\omega]\in H^1(K,\Z/3)$. By our assumption, $\omega$ is a cube in $K$, hence $[\omega]=0\in H^1(K,\Z/3)\cong K^*/(K^*)^3$. This shows $\partial^1=0$. 

To show $\partial^2=0$, consider the following commutative diagram given in \cite[Lemma~7.5.10]{MR2266528}:

\[
  \begin{tikzcd}
    \mu_3\otimes K^M_{2}(K)\arrow[swap]{d}{\omega\cup h^2_{K,3}} \arrow{r}{\{,\}} & K^M_{3}(K)/3K^M_{3}(K) \arrow[swap]{d}{h^3_{K,3}}\\
    H^2(K,\mu_3^{\otimes 3})\arrow{r}{\tilde{\partial^2}} & H^3(K,\mu_3^{\otimes 3}) 
  \end{tikzcd}
\]
Notations in the above diagram are explained in \cite[Lemma~7.5.10]{MR2266528}. Notice since $\mu_3$ are trivial $\Gamma$-modules, we have the following isomorphism (\cite[Lemma 4.1]{44267b53-f4ad-340a-8e65-df1c74888743}):
$$\phi^i_j:H^i(K,\Z/3)\cong H^i(K,\mu_{3}^{\otimes j})$$
$$\alpha\mapsto \alpha\cup\underbrace{(\omega\cup\cdots\cup \omega)}_{j \text{ copies}}$$

Since the upper horizontal map in the commutative diagram is given by the symbol product with $0=[\omega]\in H^1(K,\Z/3)$, it follows that this map is 0. By the Merkurjev-Suslin theorem (\cite[Theorem~8.6.5]{MR2266528}), the left vertical map is surjective, hence $\tilde{\partial^2}=0$. Given that cup products "commute" with boundary homomorphisms (\cite[Proposition~1.4.3]{neukirch2008cohomology}), we have the commutative diagram involving $\partial^2$:
\[
  \begin{tikzcd}
    H^2(K,\mu_3^{\otimes 3})\arrow[swap]{d}{(\phi^2_3)^{-1}}\arrow{r}{\tilde{\partial^2}} & H^3(K,\mu_3^{\otimes 3})\arrow[swap]{d}{(\phi^3_3)^{-1}} \\
    H^2(K,\Z/3)\arrow{r}{\partial^2} & H^3(K,\Z/3)
  \end{tikzcd}
\]
From this, it is clear that $\partial^2=0$.

Next notice that $K$ has characteristic $\neq 3$, so we have 
$\Br(K)[n]\cong H^2(K,\mu_n)$ when $n$ is a power of 3.
Consider the following commutative diagram:
\[
  \begin{tikzcd}
    0 \arrow{r} & \Br(K)[3]\arrow[swap]{d}{\sigma_3} \arrow{r} & \Br(K)[9] \arrow[swap]{d}{\sigma_9} \arrow{r} & \Br(K)[3] \arrow[swap]{d}{\sigma_3}\arrow{r} & 0 \\
    0\arrow{r} & \Br(S)[3]\arrow{r} & \Br(S)[9] \arrow{r} & \Br(S)[3]
  \end{tikzcd}
\]
Then the snake lemma gives us 
$$\ker(\sigma_9) \xrightarrow{\phi_1}\ker(\sigma_3)\to\coker(\sigma_3)\xrightarrow{\phi_2}\coker(\sigma_9).$$

Since $\phi_1$ is multiplication by 3 and $\ker(\sigma_9) \cong \Z/3$, $\phi_1  = 0$. The map $\phi_2$ is also zero as $\Br(K)[9]$ maps onto $\Br(S)[3]$. Hence we have 
$\coker(\sigma_3) \cong \ker(\sigma_3) \cong \Z/3.$
\end{proof}
\section{Brauer groups of Brauer-Severi surface bundles}
\label{s:3}
In this section, we present sufficient conditions under which a Brauer-Severi surface bundle 5-fold is not stably rational. These conditions are derived by generalizing \cite[Thm.~2.6]{MR4069651} to the 3-torsion case. However, the details in these two cases are quite different.

\begin{theorem}
\label{t:main}
Let $k$ be an algebraically closed field of characteristic $\neq 3$ and let $\pi:Y\to B$ be a Brauer-Severi surface bundle over a smooth projective threefold $B$ over $k$ with a smooth generic fiber. Assume $\Br(B)[3]=0$ and $H^3_{\text{ét}}(B,\Z/3)=0$. (For example, take $B=\mP^3$.) Let $\alpha\in \Br(K)[3]$ be the Brauer class over $K=k(B)$ corresponding to the generic fiber of $\pi$, and it can be represented by a cyclic algebra of index 3. Assume the discriminant locus is good [Def. \ref{d:good}] with irreducible components $S_1,\cdots,S_n$. Further suppose the following conditions also hold:
\begin{enumerate}
\item Any irreducible curve in $B$ is contained in at most two surfaces from the set $\left\{S_1,\cdots,S_n\right\}$.
\item Through any point of $B$, there pass at most three surfaces from the set $\left\{S_1,\cdots,S_n\right\}$.
\item  For all $i\neq j$, $S_i$ and $S_j$ are factorial at every point of $S_i\cap S_j$.
\end{enumerate}

Let $\gamma_i=\partial^2_{S_i}(\alpha)\in H^1(k(S_i),\Z/3)$. Let $\Gamma$ be the subgroup of $\bigoplus_{i=1}^{n}H^1(k(S_i),\Z/3)$ given by 
$\Gamma=\bigoplus_{i=1}^{n}\left<\gamma_i\right>\cong (\Z/3)^n$. We write elements of $\Gamma$ as $(x_1,x_2\cdots,x_n)$ with $x_i\in \{0,1,2\}$.

Let $H\subset \Gamma$ consist of those elements $(x_1,\cdots,x_n)\in (\Z/3)^n$ such that $x_i=x_j$ whenever there exists an irreducible component $C$ of $S_i\cap S_j$, such that 
$$(\partial^1_{C}(\gamma_i),\partial^1_{C}(\gamma_j))= (1,2)\ \textup{or} \ (2,1).$$
Let $H'\subset H$ be the subgroup consisting of elements $(x_1,\cdots,x_n)\in (\Z/3)^n$ such that $x_i=x_j$ whenever there exists an irreducible components $C$ of $S_i\cap S_j$, such that 
$$(\partial^1_{C}(\gamma_i),\partial^1_{C}(\gamma_j))=(0,0), \text{ and } 
 \gamma_i|_{C}  \text{ and }  \gamma_j|_{C} \text{ are not both trivial in } H^1(k(C),\Z/3).$$
Then $H^2_{nr}(k(Y)/k,\Z/3)$ contains the subquotient $H'/\left<1,\cdots,1\right>$.
\end{theorem}

\begin{proof}  First, note that under these assumptions, $Y$ is necessarily integral(see \ref{c:associated Brauer-Severi surface bundle}), hence we can talk about its function field $k(Y)$. 
We have the following commutative diagram:
  \[\footnotesize{ \begin{tikzcd}[row sep= 2.0 em, column sep= 2.5 em]
 & & 0 & &\\
  & & \Z/3\arrow[u] & &\\
    0 \arrow{r} & H^2_{nr}(k(Y)/Y,\Z/3) \arrow{r} & H^2_{nr}(k(Y)/K,\Z/3) \arrow{r}{\oplus \partial^2_{T}}\arrow{u}&\displaystyle\bigoplus_{T\in Y^{(1)}_{B}}H^1(k(T),\Z/3) & \\
    & \Br_{nr}(K)[3]=0  \arrow{r} & H^2(K,\Z/3) \arrow{u}{\sigma} \arrow{r}{\oplus \partial^2_{S}} &  \displaystyle\bigoplus_{S\in B^{(1)}} H^1(k(S),\Z/3) \arrow{u}{\tau} \arrow{r}{\oplus \partial^1_{C}} & \displaystyle\bigoplus_{C\in B^{(2)}} \Z/3\\
    & & \left<\alpha\right>\arrow{u} & \Gamma \arrow{u} &\\
    & &  0 \arrow{u}& 0 \arrow{u} &\\
  \end{tikzcd}
  }
  \]
We make some observations related to this diagram:
\begin{enumerate}

\item By definition, $H^2_{nr}(k(Y)/Y,\Z/3)$ denotes all those classes in $H^2(k(Y),\Z/3)$ which are unramified with respect to divisorial valuations corresponding to prime divisors on $Y$, Since the singular locus of $Y$ has codimension $\geq 2$, we can also characterize $H^2_{nr}(k(Y)/Y,\Z/3)$ as all those classes in $H^2(k(Y),\Z/3)$ that are unramified with respect to divisorial valuations which have centers on Y which are not contained in $Y_{sing}$ \cite[Cor.~2.2]{MR4069651}.

\item $H^2_{nr}(k(Y)/K,\Z/3)$ denotes those classes in $H^2(k(Y),\Z/3)$ which are killed by residue maps associated to divisorial valuations that are trivial on $K$, hence correspond to prime divisors of $Y$ dominating the base $B$. We use $Y_{B}^{(1)}$ to denote all prime divisors in $Y$ that do not dominate the base $B$. Then the upper row is exact by definition.

\item The second row is obtained from Bloch-Ogus complex \cite{MR412191}, which is exact under the assumptions 
$$\Br(B)[3]=0,\textup{and} \;  H^3_{\text{ét}}(B,\Z/3)=0.$$

\item The left vertical row is exact by Lemma \ref{l:ses}, because we have 
$$H^2_{nr}(k(Y)/K,\Z/3)\cong H^2_{nr}(K(S_0)/K,\Z/3)\cong Br(S_0)[3],$$
where $S_0$ is the Brauer-Severi surface (over $K$) corresponding to the generic fiber $\alpha$. Hence $S_0$ is smooth and we have the last isomorphism in the above statement.

\item In the right vertical row, the map $\tau$ is induced by the field extensions
$k(S)\subset k(T)$, coincides with the induced map
$$k(S)^{\times}/k(S)^{\times 3}\to k(T)^{\times}/k(T)^{\times 3}.$$
If $S$ is not contained in the discriminant locus, the generic fiber of $T\to S$ is geometrically integral. Then $k(S)$ is algebraically closed in $k(T)$, and thus the induced map above is injective. If $S=S_i$ is a component of the discriminant locus, then after taking the base change to the cubic extension $F/k(S_i)$ defined by the residue class $\gamma_i\in H^1(k(S_i),\Z/3)$, the generic fiber of $T_i\to S_i$ is a union of three Hirzebruch surfaces $\mathbb{F}_1$, meeting transversally so that any pair of them meet along a fiber of one and a $(-1)$-curve of the other. (This is correct because $T_i$ is indeed the preimage of $S_i$ under $\pi$, with the third assumption in Definition \ref{d:good}).  In this case, the low degree long exact sequence from the Hochschild-Serre spectral sequence 
$$H^p(\Gal(F/k(S_i)),H^q(\Spec(F),\Z/3))\Rightarrow H^{p+q}(\Spec(k(S_i)),\Z/3)$$
implies the kernel of the natural map $H^1(k(S_i),\Z/3)\to H^1(F,\Z/3)$ is generated by $\gamma_i$.  By the last assumption in Definition \ref{d:good}, we know $\ker(\tau)=\Gamma$.
\end{enumerate}
Then we can prove that $H^2_{nr}(k(Y)/Y,\Z/3)$ lies in the image of $\sigma$.  In fact, let $\xi\in H^2_{nr}(k(Y)/Y,\Z/3)$ and denote by $\xi$ again its the image in $H^2_{nr}(k(Y)/Y,\Z/3)$. Then $\xi$ is killed by $\oplus \partial^2_{T}$. If $\xi$ is not in the image of $\sigma$, it lifts to a class $\xi^{'}\in H^2(K,\Z/9)$ by Lemma \ref{l:ses}. We have the following exact sequence which is similar to the second row in above diagram with coefficients $\Z/9$:
$$0  \to H^2(K,\Z/9)  \xrightarrow{\oplus \partial^2_{S}}   \displaystyle\bigoplus_{S\in B^{(1)}} H^1(k(S),\Z/9) \xrightarrow{\oplus \partial^1_{C}} \displaystyle\bigoplus_{C\in B^{(2)}} \Z/9\\$$
Hence at least one residue $\partial^2_{S}(\xi^{'})$ must have order 9 (since $\oplus\partial^2_{S}$ is injective both for 3-torsion and 9-torsion cases ). On the other hand, 

$$\ker (\bigoplus_{S\in B^{(1)}} H^1(k(S),\Z/9)\to \bigoplus_{T\in Y_{B}^{(1)}}H^1(k(T),\Z/9))\cong \Gamma$$

This is correct because (again we use $F$ to denote a separable closure of $k(S_i)$) in the long exact sequence associate to 
$$H^p(\Gal(F/k(S_i)),H^q(\Spec(F),\Z/9))\Rightarrow H^{p+q}(\Spec(k(S_i)),\Z/9)$$
We have 
$$0\to H^1(Gal(F/k(S_i)),\Z/9)\to H^1(\Spec(k(S_i)),\Z/9)\to H^1(\Spec(F),\Z/9)$$
As we can calculate \`etale cohomology of spectrum of a field using Galois cohomology, we have the kernel: 
$$H^1(Gal(F/k(S_i)),\Z/9)\cong Hom_{cont}(\Z/3,\Z/9)\cong \Z/3$$
While the kernel for those $S$ doesn't belongs to the discriminant locus is clearly zero by the same argument in 3-torsion case.

Now we notice that $\Gamma$ has no elements of order 9, this means $\partial^2_{S}(\xi^{'})$ can't be mapped to 0 in $\bigoplus_{T\in Y^{(1)}_{B}}H^1(k(T),\Z/3)$,
hence a contradiction.

The above diagram chasing in fact gives us 
$$H^2_{nr}(k(Y)/Y,\Z/3)\cong \Gamma\cap \ker(\oplus \partial^1_{C})/\left<\alpha\right> \cong H/\left<\alpha\right>.$$

Next we determine classes in $H$ that are in $H^2_{nr}(k(Y)/k,\Z/3)$. In particular, we show that the subgroup $H'$ defined earlier is contained in $H^2_{nr}(k(Y)/k,\Z/3)$. We do this by checking whether the classes in $H'$ are unramified with respect to all divisorial valuations $\mu$ of $k(Y)$ (and not just those that come from prime divisors on $Y$). 
Consider a class $\beta\in H$, viewed as an element in $H^2(K,\Z/3)$. Denote by $\beta'$ the image of $\beta$ in $H^2(k(Y),\Z/3)$. We aim to show that $\beta'$ is unramified on $Y$ if $\beta$ is in $H'$. Using the definition of $H$, it is sufficient to check this for valuations whose centers on $B$ has codimension at least 1. In the following, we use $\mathscr{O}$ to denote the local ring of $\mu$ in $B$.

\begin{enumerate}
\item[\bf Case 1:] {\it The center of $\mu$ on $B$ is not contained in the discriminant locus:} In this case, for any surface $S$ passing through the center of $\mu$, we have 
$$\partial^2_{S}(\beta)=0.$$
Then \cite[Proposition~2.1]{MR4069651} tells us $\beta$ is in the image of $H^2_{\text{ét}}(\mathscr{O},\Z/3)$. Hence $\beta'=\sigma(\beta)$ is also unramified  with respect to $\mu$ in this case.

\item[\bf Case 2:] {\it The center of $\mu$ on $B$ is contained in the discriminant locus, but not in the intersection of two or more components:} Now the center is contained in $S_i$ for a unique $i$. Recall that the $i^{\rm th}$ component $x_i$ of $\oplus \partial^2_{S}(\beta)$ is $0,1$ or $2$. If $x_i=0$, by an argument same as Case 1, $\beta$ is in the image of $H^2_{\text{ét}}(\mathscr{O},\Z/3)$. Similarly, if $x_i=1$, $\beta-\alpha$ is in the image of $H^2_{\text{ét}}(\mathscr{O},\Z/3)$. Finally, if $x_i=2$, $\beta-2\alpha$ is in the image of $H^2_{\text{ét}}(\mathscr{O},\Z/3)$. Notice that 
$$\beta'=\sigma(\beta)=\sigma(\beta-\alpha)=\sigma(\beta-2\alpha).$$
So in all three conditions, we have $\beta'$ is unramified with respect to $\mu$ in this case.

\item[\bf Case 3:] {\it The center $\mu$ on $B$ is a curve $C$ that is an irreducible component of $S_i\cap S_j$:} In this case, we again check the possible values of $x_i$ and $x_j$ in $\oplus \partial^2_{S}(\beta)$. We have the following cases:
\begin{itemize}
\item[Case 3(a):] If $x_i=x_j$, then the argument in Case 2 above gives us that at least one of $\beta, \beta-\alpha$ or $\beta-2\alpha$ lies in the image of $H^2_{\text{ét}}(\mathscr{O},\Z/3)$. So we are done in this situation.

\item[Case 3(b):] {\it $(x_i,x_j)=(0,1)$ or $(1,0)$:} By symmetry, we can assume $(x_i,x_j)=(1,0)$. Notice that
$$3\left|\left(\partial^1_C(\gamma_i)+\partial^1_C(\gamma_j)\right)\right.$$
by the exactness of the second row in the diagram. Then we must have $$\partial^1_C(\gamma_i)=\partial^1_C(\gamma_j)=0$$

This means that a rational function representing the class $$\gamma_i\in H^1(k(S_i),\Z/3)=k(S_i)^{\times}/k(S_i)^{\times 3}$$ has a zero or a pole of order divisible by 3 along $C$. Without loss of generality, we may assume that the function associated with $\gamma_i$ is contained in the local ring $\mathscr{O}_{S_i,C}$ of $C$ in $S_i$.We call this function $f_{\gamma_i}$. Let $t$ be a local parameter for $C$ in $\mathscr{O}_{S_i,C}$. Such a local parameter exists as $C$ is a Cartier divisor on $S_i$, which in turn follows since $S_i$ is assumed to be factorial along $C$. Then 
$\displaystyle \frac{f_{\gamma_i}}{t^{\mu_C(f_{\gamma_i})}}$
is a unit, and hence any preimage in $\mathscr{O}$ is also a unit (See Remark \ref{r:3b} below). Call such a preimage $u_{\gamma_i}$, which may be viewed as a rational function in $K$. Assume $\pi_{S_i}$ is a local parameter of $S_i$ in $\mathscr{O}$. Consider the symbol algebra $(u_{\gamma_i},\pi_{S_i})\in H^2(K,\Z/3)$. Let $S$ be a surface containing $C$. By lemma \ref{l:residue}, we have
\begin{align*}
    \partial^2_{S}(u_{\gamma_i},\pi_{S_i})=(-1)^{\mu_S(u_{\gamma_i})\mu_S(\pi_{S_i})}{\frac{\bar{u}_{\gamma_i}^{\mu_S(\pi_{S_i})}}{\pi_{S_i}^{\mu_{S}(u_{\gamma_i})}}}=\left\{
    \begin{array}{cl}
    \bar{u}_{\gamma_i} & \text{if} \ S=S_i \\
    0 & \text{if} \ S\neq S_i \\
    \end{array} \right.
\end{align*}

On the other hand, $\bar{u}_{\gamma_i}=\gamma_i$ by construction, so we have 
$$\partial^2_{S_i}(u_{\gamma_i},\pi_{S_i})=\gamma_i=\partial^2_{S_i}(\beta)$$
$$\partial^2_{S_j}(u_{\gamma_i},\pi_{S_i})=0=\partial^2_{S_j}(\beta)$$
Also $\partial^2_{S}(\beta)=0$ if $S$ is a surface passing through $C$ other than $S_i$ and $S_j$.(In fact, by our assumption, such an $S$ is not in the discriminant locus and so this agrees with this conclusion.) Hence \cite[Proposition~2.1]{MR4069651} tells us 
$\beta-(u_{\gamma_i},\pi_{S_i})$ is in the image of $H^2_{\text{ét}}(\mathscr{O},\Z/3)$. Hence
$$\partial^2_{\mu}(\sigma(\beta-(u_{\gamma_i},\pi_{S_i})))=0$$
It then suffices to show that 
$$\partial^2_{\mu}(\sigma((u_{\gamma_i},\pi_{S_i})))= \pm \bar{u}_{\gamma_i}^{\mu(\pi_{S_i})}=0 \in H^1(k(\mu),\Z/3)$$
By assumption,  $\bar{u}_{\gamma_i}|_{C}$ is trivial, hence so is $\bar{u}_{\gamma_i}^{\mu(\pi_{S_i})}$ as the center of $\mu$ is $C$.

\item[Case 3(c):] {\it $(x_i,x_j)=(0,2)$ or $(2,0)$:} By symmetry, we can assume $(x_i,x_j)=(2,0)$. Now the proof is essentially same as Case 3(b), which shows that  
$$\partial^2_{\mu}(\sigma(\beta-2(u_{\gamma_i},\pi_{S_i})))=0.$$
It follows that $\partial^2_{\mu}(\sigma(\beta))=0.$ and so $\beta'$ is unramified along $\mu$.
\item[Case 3(d):] {\it $(x_i,x_j)=(1,2)\  \text{or}\  (2,1)$:} Assume $(x_i,x_j)=(2,1)$. Then  applying Case $3(b)$ to the class  $\beta-\alpha$, we see that this case is also proved.
\end{itemize}
\item[\bf Case 4:] {\it The center of $\mu$ on $B$ is a point $P\in C$, here $C$ is as in case 3, and $S_i,S_j$ are the only surfaces among the $S_1,\cdots,S_n$ that pass through $P$.}  As we have seen in the discussion of Case 3, we can reduce to the case when $(x_i,x_j)=(1,0)$. Hence we again have 
$\partial^1_{C}(\gamma_i)=\partial^1_{C}(\gamma_j)=0$.
In fact, this is true for any curve $C'$ that contains $P$ and is contained in $S_i\cup S_j$.  Choose a function $f_{\gamma_i}\in k(S_i)$ representing the class $\gamma_i$. Then let $C_1,\cdots,C_N$ be all irreducible curves through $P$ that are either a zero or a pole for the function $f_{\gamma_i}$. Pick local equations $t_{\ell}$ of $C_{\ell}$ in $\mathscr{O}_{S_i,P}$, and consider the following rational function on $S_i$: 
$$\frac{f_{\gamma_i}}{\left(t_1^{\mu_{C_1}(f_{\gamma_i})}\cdots t_N^{\mu_{C_N}(f_{\gamma_i}) }\right)}.$$
Since $S_i$ is assumed to be factorial, in particular, normal at $P$, the above rational function is a unit locally around $P$. Hence it can be lifted to a unit in $\mathscr{O}$. Then we can repeat the rest of the proof as in Case 3(b). (Notice that every element in $k(P)$ is a cube since $k$ is algebraically closed, so the last step of Case 3(b) is automatically true.)

\item[\bf Case 5:] {\it The center of $\mu$ on $B$ is a point $P$ that lies on exactly three distinct surfaces $S_i, S_j, S_l$:} 
We consider the possible values of $(x_i,x_j,x_l)$. If $x_i=x_j=x_l$, then one of $\beta,\beta-\alpha$ or $\beta-2\alpha$ is unramified. By symmetry and up to subtraction by $\alpha$ or $2\alpha$, the only remaining cases are $(1,0,0)$, $(1,1,0)$, and $(2,1,0)$. Notice that $(2,1,0) = (1,1,0) + (1,0,0)$, and that the case $(1,1,0)$ is equivalent to the case $(2,0,0)$. Hence, we only need to consider the case $(1,0,0)$, which is same as Case 4. Now the rest of the proof is same as in Case 4.
\end{enumerate}
\end{proof}
\begin{remark}
\label{r:3b}
    In Case 3(b) in Theorem \ref{t:main}, we claimed that if  $\bar{x}\in \mathscr{O}_{S_i,C}$ is a unit, then any preimage $x$ in $\mathscr{O}=\mathscr{O}_{B,C}$ is also a unit. In fact, we have 
    $$\mathscr{O}_{S_i,C} \cong \mathscr{O}/(\pi_{S_{i}}).$$
As $\bar{x}$ is a unit in $\mathscr{O}_{S_i,C}$, there exist a $\bar{y}\in \mathscr{O}_{S_i,C}$ such that $\bar{x}\bar{y}=1\in \mathscr{O}_{S_i,C}$. Hence there exist $t\in \mathscr{O}$, such that 
$$xy=1+\pi_{S_{i}}t\in \mathscr{O}$$
Notice that $\pi_{S_i}t$ is contained in the maximal ideal of $\mathscr{O}$, so $1+\pi_{S_{i}}t$ is a unit in $\mathscr{O}$. Hence any preimage $x$ is also a unit in $\mathscr{O}$.
\end{remark} 

We prove an immediate corollary in which we weaken the hypothesis about factoriality when $n=2$. In this case, the discriminant locus has exactly two irreducible components. We prove that it is sufficient to have only one of them factorial at their intersection to make the unramified Brauer group nontrivial:
\begin{cor}
\label{c:main}
Assume $n = 2$. We continue with the same hypothesis as in the theorem except the following change: we replace the requirement $(3)$ by  the following:
\\
(3') $S_1$ is factorial at every point of $S_1\cap S_2$ .
\\
Then  $H^2_{nr}(k(Y)/k,\Z/3)$ is nontrivial and hence $Y$ is not stably rational.
\end{cor}
\begin{proof}
In this case, the Brauer class $\beta$ in $H'$ whose representative is $(1,0)$ can be lifted to a nontrivial unramified Brauer class in $H^2(k(Y)/k,\Z/3)$
 \end{proof}

 \section{Example}
 \label{s:4}
 In this section, we will construct a Brauer-Severi surface bundle over $\mP^3$ that is stably non-rational. We use Corollary \ref{c:main} for this purpose. 

 \begin{example}
 \label{e}
Consider the following two surfaces in $\mP^3_{\C}={\rm Proj} \ \C[x_0,x_1,x_2,x_3]$:
 $$S_1:\{x_0^9+(x_1^3-x_2^3)(x_2^3-x_3^3)(x_3^3-x_1^3)=0\}$$
 $$S_2:\{\left(x_0^9+(x_1^3-x_2^3)(x_2^3-x_3^3)(x_3^3-x_1^3)\right)\left(x_0^9-x_1^3x_2^3x_3^3\right)+x_1^6x_2^6x_3^6=0\}$$
 In the following, we use $F_{S_1},F_{S_2}$ to denote the equation defines $S_1$, $S_2$ separately. We start by checking that both $S_1$ and $S_2$ are irreducible and reduced:
\begin{itemize}

    \item {\it $S_1$ is irreducible and reduced.} This follows directly from the fact that the singular locus of $S_1$ has dimension 0. In fact,  $S_1$ only singular at 12 isolated points:
    $$[0:1:0:0]\ ,\ [0:0:1:0]\ ,\ [0:0:0:1]$$
    $$[0:\omega:1:1]\ ,\ [0:1:\omega:1]\ ,\ [0:1:1:\omega]$$
    $$[0:\omega^2:1:1],[0:1:\omega^2:1],\ [0:1:1:\omega^2]$$
    $$[0:\omega^2:\omega:1]\ ,[0:\omega:\omega^2:1]\ ,[0:1:1:1]$$
    Here $\omega$ is a primitive $3^{rd}$ roots of unity. If $S_1$ is not reduced, then the singular locus would have dimension 2. If $S_1$ is not irreducible, the singular locus would have dimension at least 1 by B\`ezout theorem.
\item {\it $S_2$ is irreducible and reduced.} We may rewrite the equation defining $S_2$ as:
$$x_0^{18}+P(x_1,x_2,x_3)x_0^9-x_1^3x_2^3x_3^3P(x_1,x_2,x_3)$$
where $P(x_1,x_2,x_3)=(x_1^3-x_2^3)(x_2^3-x_3^3)(x_3^3-x_1^3)-x_1^3x_2^3x_3^3$. We may consider the above polynomial as an element in $\C[x_0,x_1,x_2,x_3]=\C[x_1,x_2,x_3][x_0]$, which is a UFD. Hence to check it is irreducible, it is sufficient to use Eisenstein's criterion: We need to find a prime ideal $\mathfrak{p}$ in $\C[x_1,x_2,x_3][x_0]$, such that $$P(x_1,x_2,x_3)\in \mathfrak{p},$$
$$x_1^3x_2^3x_3^3P(x_1,x_2,x_3)\in \mathfrak{p}\  \text{and} $$  $$x_1^3x_2^3x_3^3P(x_1,x_2,x_3)\notin \mathfrak{p}^2.$$
It is evident that an appropriate prime ideal exists if $P(x_1,x_2,x_3)$ has an irreducible factor with multiplicity 1 and is coprime to $x_1x_2x_3$. In fact, any irreducible factor of $P(x_1,x_2,x_3)$ is inherently coprime to $x_1x_2x_3$. Therefore it suffices to provide a single regular point of $P(x_1,x_2,x_3)$ to show the existence of such an irreducible factor. Finally, we directly check that $(1,1,0)$ is a regular point of $P(x_1,x_2,x_3)$. Hence $S_2$ is irreducible.

Now as we have already shown $S_2$ is irreducible, it is sufficient to find a smooth point in $S_2$ to show it is reduced. Indeed, one can easily check that $[(-56)^{\frac{1}{9}}:1:2:0]$ is indeed a smooth point of $S_2$.
\end{itemize}
We choose rational triple covers of $S_1$ and $S_2$ defined by: $$\gamma_1=\overline{\frac{x_2^3-x_3^3}{x_0^3}}\in H^1(\C(S_1),\Z/3)\cong \C(S_1)^{\times}/\C(S_1)^{\times 3}$$
 $$\gamma_2=\overline{\frac{x_0^9-x_1^3x_2^3x_3^3}{x_0^9}}\in H^1(\C(S_2),\Z/3)\cong \C(S_2)^{\times}/\C(S_2)^{\times 3}$$
We claim the triple covers $\gamma_1, \gamma_2$ are not trivial: In fact, by Lemma \ref{l:residue}, the residue of $\gamma_1$ of a valuation centered at the point $[0:\omega:1:1]$ is $1\in \Z/3$. Hence $\gamma_1$ is not trivial. To show $\gamma_2$ is not trivial is equivalent to show $F_{S_1}$ is not a cubic in the function field of $S_2$. And it's true because the residue of $\overline{\frac{F_{S_1}}{x_0^9}}$ of a valuation centered at the point $[0:0:1:1]$ is $1\in \Z/3$. Hence $\gamma_1, \gamma_2$ are not trivial.

 Consider the corresponding Bloch-Ogus exact sequence:   
  \[ 
  \begin{tikzcd}[sep=small]
    & 0 \arrow{r} & Br(\C(\mP^3_{\C}))[3]  \arrow{r}{\oplus \partial^2_{S}} &  \displaystyle\bigoplus_{S\in (\mP^3_{\C})^{(1)}} H^1(k(S),\Z/3)  \arrow{r}{\oplus \partial^1_{C}} & \displaystyle\bigoplus_{C\in (\mP^3_{\C})^{(2)}} H^0(k(C),\Z/3)\\
  \end{tikzcd}
  \]
   We have $\oplus \partial^1_{C}(\gamma_1)=\oplus \partial^1_{C}(\gamma_2)=0$. In fact, it is easy to check that for any curve $C$ such that $\gamma_1$(or $\gamma_2$) has a zero or pole along $C$, the order is divided by 3. Hence
  $$(1,\cdots,1,\gamma_1,1,\cdots,1,\gamma_2,1,\cdots)\in  \displaystyle\bigoplus_{S\in (\mP^3_{\C})^{(1)}} H^1(k(S),\Z/3)$$
  can be lifted to a Brauer class $[\mathscr{A}]=[(\frac{F_{S_2}(x_2^3-x_3^3)}{x_0^{21}},\frac{F_{S_1}}{x_0^9})_{\omega}] \in \Br(\C(\mP^3_{\C}))[3]$. This can be directly checked by Lemma \ref{l:residue}. 

By Theorem \ref{t:flatness}, the cyclic algebra $\mathscr{A}$ gives out a Brauer-Severi surface bundle $Y\to \mP^3_{\C}$. This Brauer-Severi surface bundle has a good discriminant locus. We prove this by checking the conditions in Definition \ref{d:good}. Here is the list of corresponding arguments:
\begin{enumerate}
    \item We already proved that $S_1$ and $S_2$ are reduced.
    \item The behavior of a general fiber over $S_1$ and $S_2$ is given by \cite[Thm.~2.1]{MR1480776}.
    \item The induced triple cover over $S_1$ and $S_2$ are irreducible because $\gamma_1$, $\gamma_2$ are not trivial.
    \item To show the last requirement in Definition \ref{d:good} is true, we have the following commutative diagram:
\[ 
  \begin{tikzcd}[row sep= 1 em, column sep= 1 em]
F^{\times}/F^{\times 3}\arrow[r,"a"] & F(u,v)^{\times}/F(u,v)^{\times 3}\\
k(S_i)^{\times}/k(S_i)^{\times 3} \arrow[u,"b"]\arrow[r,"\tau_i"] & k(T_i)^{\times}/k(T_i)^{\times 3}\arrow[u,"d"]\\
  \end{tikzcd}
  \]
Where $T_i$ is defined right after the large diagram in Theorem \ref{t:main}. 
For $S_1$, $b$ is induced by the cubic extension defined by $\gamma_1$, $d$ is induced by the cubic extension defined by $\frac{F_{S_2}(x_2^3-x_3^3)}{x^{21}_0}$, which is equal to the cubic class defined by $\gamma_1$ (\cite[Thm.~2.1]{MR0657429}. )
Note that $a$ is injective, and  $\ker(b)=<\gamma_1>$. On the other hand, an easy diagram chasing as in part (5) in proof of Theorem \ref{t:main} shows that $\ker(\tau_1)$ contains $<\gamma_1>$ . This forces $d$ to be injective and $\ker(\tau_1)=<\gamma_1>$.
Same argument works for $S_2$.
\end{enumerate}
 On the other hand, we list all irreducible components of $S_1\cap S_2$:
  $$D=6D_1+6D_2+6D_3$$
  Where
  $$D_1=\{x_1=0,x_0^9-x_2^6x_3^3+x_2^3x_3^6=0\}$$
    $$D_2=\{x_2=0,x_0^9-x_3^6x_1^3+x_3^3x_1^6=0\}$$
      $$D_3=\{x_3=0,x_0^9-x_1^6x_2^3+x_1^3x_2^6=0\}$$
One can easily check they are indeed irreducible using Eisenstein's criterion by viewing those polynomials as elements in $\C[x_1,x_2,x_3][x_0]$.
  Notice that $D_1$ passes through only two singular points of $S_1$: $[0:0:1:0]$ and $[0:0:0:1]$. It is straightforward to check $D_1$ is indeed a Cartier divisor of $S_1$, even along these two singular points:
\begin{lemma}
$D_i$ are Cartier divisors of $S_1$.
\end{lemma}
\begin{proof}
By symmetry, it is sufficient to check the behavior of $D_1$ at singular points of $S_1$. Notice that $D_1$ only passes through two singular points of $S_1$: $[0:0:1:0]$ and $[0:0:0:1]$. Let $P=[0:0:0:1]$, then we have the local ring:
$$\mathscr{O}_{S_1,P}=(\C[x_0,x_1,x_2]/(x_0^9+(x_1^3-x_2^3)(x_2^3-1)(1-x_1^3)) )_{(x_0,x_1,x_2)}$$
By expanding the equation defining $S_1$, we have 
$$x_0^9-x_2^6+x_2^3=x_1^3(x_1^3x_2^3-x_2^6-x_1^3+1)\in \mathscr{O}_{S_1,P}$$
Notice $x_1^3x_2^3-x_2^6-x_1^3+1$ is a unit in $\mathscr{O}_{S_1,P}$, hence the ideal defining $D_1$, which is $(x_1,x_0^9-x_2^6+x_2^3)$, is generated by one element $x_1$. Similarly one can do the calculation for the point $[0:0:1:0]$. As a result, $D_1$ is a Cartier divisor of $S_1$.

\end{proof}

Finally, we need to check both $\gamma_1|_{D_i}$ and $\gamma_2|_{D_i}$ are trivial for $i\in\{1,2,3\}$. These are directly following from the choices of $\gamma_1$ and $\gamma_2$. Hence in this example, using notations in Theorem \ref{t:main}, we have $H'=H=\Gamma=\Z/3\times\Z/3$. By Corollary \ref{c:main}, the unramified Brauer group of $Y$ contains a subgroup $\Z/3$, hence $Y$ is not stably rational.
\end{example}

\section{Flatness}
\label{s:5}
In this section, we check the cyclic algebra 
$$\mathscr{A}=(\frac{F_{S_2}(x_2^3-x_3^3)}{x_0^{21}},\frac{F_{S_1}}{x_0^9})_{\omega}$$
indeed gives us a Brauer-Severi surface bundle over $\mP^3_{\C}$ as in Definition \ref{d:bundle}. We keep the notation in Example \ref{e} through out this section. The definition of a general Brauer-Severi scheme is given by Van den Bergh in \cite{MR1048420}. In \cite{MR1710744}, Seelinger gave an alternating description of Brauer-Severi scheme which is easier to use in our case. See also Section 1 in \cite{MR1480776} for the discussion of the following definitions:
\begin{definition}
\label{d:order}
Let $\Lambda$ be a sheaf of $\mathscr{O}_{\mP^3_{\C}}$ algebra that is torsion free and coherent as an $\mathscr{O}_{\mP^3_{\C}}$ module. We say $\Lambda$ is an $\mathscr{O}_{\mP^3_{\C}}$-order in $\mathscr{A}$ if $\Lambda$ contains $\mathscr{O}_{\mP^3_{\C}}$ and 
$$\Lambda\otimes_{\mathscr{O}_{\mP^3_{\C}}}\C(\mP^3_{\C})\cong \mathscr{A}$$
\end{definition}
\begin{definition}
\label{d:local order}
    For each point $p\in \mP^3_{\C}$, let $\mathscr{O}_{\mP^3_{\C},p}$ denote the regular local ring of $\mP^3_{\C}$ at $p$. We say a finitely generated $\mathscr{O}_{\mP^3_{\C},p}$ algebra $\Lambda_{p}$ is an $\mathscr{O}_{\mP^3_{\C},p}$-order in $\mathscr{A}$, if $\Lambda_{p}$ is torsion free and 
    $$\Lambda_{p}\otimes_{\mathscr{O}_{\mP^3_{\C},p}}\C(\mP^3_{\C})\cong \mathscr{A}$$
    \begin{remark}
    In this paper, we always assume an order is locally free.
    \end{remark}
\end{definition}
Recall that $(3.4)$ of \cite{MR1480776} describes an $\mathscr{O}_{\mP^3_{\C}}$-order which we again denote by $\Lambda$ in the following, we denote its localization at a point $p$ by $\Lambda_{p}$.
\begin{definition}
\label{d:Brauer-severi scheme}
Let $V_{\Lambda}$ (respectively, $V_{\Lambda_p}$) be the functor from the category of $\mP^3_{\C}$-schemes (respectively, $\Spec(\mathscr{O}_{\mP^3_{\C},p})$-schemes) to the category of sets:
$$V_{\Lambda}(S)=\{[z]\in G_n[(\Lambda\otimes_{\mathscr{O}_{\mP^3_{\C}}}S)^{\vee}]\ |\ z\cdot u=N_{S}(u)z\ ,\  \forall u\in(\Lambda\otimes_{\mathscr{O}_{\mP^3_{\C}}}S)^{*} \}$$
$$V_{\Lambda_p}(S)=\{[z]\in G_n[(\Lambda_p\otimes_{\mathscr{O}_{\mP^3_{\C},p}}S)^{\vee}]\ |\ z\cdot u=N_{S}(u)z\ ,\  \forall u\in(\Lambda\otimes_{\mathscr{O}_{\mP^3_{\C},p}}S)^{*} \}$$
where $\vee$ denotes the dual sheaf, $*$ denotes the unit group , $N_{S}$ is the reduced norm and $G_n$ denotes the functor of Grassmannian of $n$-quotients(\cite[Def.1]{MR1048420}). These functors are represented by schemes as these are closed subschemes of the Grassmannian, which we call  the Brauer-Severi scheme (associated to $\Lambda$, $\Lambda_p$) and again denote them by $V_{\Lambda}$, $V_{\Lambda_p}$. 
\end{definition} 

\begin{theorem}
\label{t:flatness}
$Y=V_{\Lambda}$ is a Brauer-Severi surface bundle over $\mP^3_{\C}$. 
\end{theorem}
\begin{proof}
According to Definition \ref{d:Brauer-severi scheme}, for every closed point $p$ in $\mP^3_{\C}$,  we have the following commutative diagram of schemes:
\[ 
  \begin{tikzcd}[row sep= 1 em, column sep= 1 em]
V_{\Lambda}\times_{\mP^3_{\C}}\Spec(\mathscr{O}_{\mP^3_{\C},p})\cong V_{\Lambda_p}\arrow[d,"\pi_p"]\arrow[r] & V_{\Lambda}\arrow[d,"\pi"]\\
 \Spec(\mathscr{O}_{\mP^3_{\C},p}) \arrow[r] & \mP^3_{\C}\\
  \end{tikzcd}
  \]
We first show $\pi$ is a flat morphism. In order to do so, it suffices to show $\pi_p$ is flat for all closed points $p\in \mP^3_{\C}$. Indeed, if this is done, the flat locus of $\pi$ would be an open subset of $\mP^3_{\C}$ containing all closed points, hence is equal to $\mP^3_{\C}$. Furthermore, by the "Miracle flatness" theorem \cite{stacks-project} and the fact that $\Spec(\mathscr{O}_{\mP^3_{\C},p})$ is regular, it suffices to show each $V_{\Lambda_p}$ is Cohen-Macaulay and each fiber of $\pi_p$ has the same dimension. We do this by a case-by-case argument for all closed points in $\mP^3_{\C}$:
\begin{enumerate}
    \item[\bf Case 1:] {\it $p\notin S_1\cup S_2$.} It is well know that $\Lambda$ is an Azumaya algebra outside of discriminant locus \cite{MR0657429}. All fibers of $\pi_p$ are smooth Brauer-Severi surfaces and furthermore $V_{\Lambda_p}$ is regular, hence Cohen-Macaulay. By the "Miracle flatness" theorem, $\pi_p$ is flat in this case.
\item[\bf Case 2:] {\it $p\in S_1\cup S_2$ and $p\notin S_1\cap S_2$ and $p\notin S_1\cap \{x_2^3-x_3^3=0\}$ .} 
Following ideas from Artin \cite{MR0657429}, we may write $\Lambda_{p}$ as the symbol algebra $(f_p,g_p)_{\omega}$. That is, $\Lambda_{p}$ over $\Spec(\mathscr{O}_{\mP^3_{\C},p})$ is generated by $x, y$ subject to the relations $$x^3=f_p, y^3=g_p, xy=\omega yx.$$ 
Since $p\notin S_1\cap \{x_2^3-x_3^3=0\}$, it follows that $f_p$ is a unit in $\mathscr{O}_{\mP^3_{\C},p}$. Let $R_p=\mathscr{O}_{\mP^3_{\C},p}[T]/(T^3-f_p)$, then 
$$\Spec(R_p)\xrightarrow[]{\tau} \Spec(\mathscr{O}_{\mP^3_{\C},p})$$
is an \'etale neighborhood of $\Spec(\mathscr{O}_{\mP^3_{\C},p})$ with $\tau$ faithfully flat as it surjects on the underlying topological space. By faithfully flat descent, it suffices to show $V_{\Lambda_p}\otimes R_p$ is flat over $\Spec(R_p)$. In \cite{MR0657429}, Artin noticed $V_{\Lambda_p}\otimes R_p$ can be viewed as a subalgebra of the 3 by 3 matrices algebra over $R_p$ by setting
$$x=
\begin{bmatrix}
     T      & 0 & 0 \\
    0       & T\omega & 0 \\
    0      & 0 & T\omega^2 
\end{bmatrix}, y= \begin{bmatrix}
     0      & 1 & 0 \\
    0       & 0 & 1 \\
    g_p      & 0 & 0 
\end{bmatrix}$$
And $V_{\Lambda_p}\otimes R_p$ can be embedded into $\mP^2_{R_p}\times\mP^2_{R_p}\times \mP^2_{R_p}$ by the following 9 equations with a cyclic permutations in indices:
$$g_p\xi_{11}\xi_{22}=\xi_{12}\xi_{21}$$
$$g_p\xi_{11}\xi_{23}=\xi_{13}\xi_{21}$$
$$g_p\xi_{11}\xi_{32}=g_p\xi_{12}\xi_{31}$$
$$g_p\xi_{11}\xi_{33}=\xi_{13}\xi_{31}$$
$$\xi_{12}\xi_{23}=\xi_{13}\xi_{22}$$
$$g_p\xi_{12}\xi_{33}=\xi_{13}\xi_{32}$$
$$g_p\xi_{21}\xi_{32}=g^2_p\xi_{22}\xi_{31}$$
$$g_p\xi_{21}\xi_{33}=g^2_p\xi_{23}\xi_{31}$$
$$g_p\xi_{22}\xi_{33}=\xi_{23}\xi_{32}$$
 Here we use $[\xi_{11}:\xi_{12}:\xi_{13}],[\xi_{21}:\xi_{22}:\xi_{23}],[\xi_{31}:\xi_{32}:\xi_{33}]$ to denote the coordinates in $\mP^2_{R_p}\times\mP^2_{R_p}\times \mP^2_{R_p}$. Note that even though Artin's original calculation assume the local ring is a DVR, \cite[Prop~3.6]{MR0657429} does work for any regular local rings \cite[Thm~2.1]{MR1480776}. If $g_p$ is part of a regular system of parameters of $R_p$, then sections 4 of \cite{MR0657429} tells us  $V_{\Lambda_p}\otimes R_p$ is indeed regular. If $p$ is a singular point of $S_1$ or $S_2$ which doesn't lie in $S_1\cap S_2$), $V_{\Lambda_p}\otimes R_p$ is not regular. However, from the above equations, a direct calculations show that on each standard affine chart (e.g. $\{\xi_{11}=\xi_{21}=\xi_{31}=1\}$),  $V_{\Lambda_p}\otimes R_p$ can be defined by 4 equations. Hence $V_{\Lambda_p}\otimes R_p$ has an open cover with each a complete intersection in $\mA^6_{R_p}$, furthermore the coordinate ring of each affine chart is again a complete intersection as a $\C$-algebra by counting dimensions. Hence $V_{\Lambda_p}\otimes R_p$ is Cohen-Macaulay. 

Consider the points (not necessarily  closed) $q\in \Spec(R_p)$. If $g_p\in m_q$, the fiber over $q$ is the union of three standard Hirzebruch surfaces $\mathbb{F}_1$, meeting transversally, such that any pair of them meet along a fiber of one and the $(-1)$-curve of the other (\cite[Prop.~3.10]{MR0657429}). If $g_p\notin m_q$, the fiber over $q$ is completely determined by $[\xi_{11}:\xi_{12}:\xi_{13}]$, hence is isomorphic to $\mP^2_{R_p}$. So, in particular, the closed fiber is the union of three $\mathbb{F}_1$ as desired and all fibers have same relative dimension. Again by the "Miracle flatness" theorem,  $V_{\Lambda_p}\otimes R_p$ is flat over $\Spec(R_p)$. As $\tau$ is faithfully flat, $\pi_p$ is flat in this case.
\item[\bf Case 3:] {\it $p\in S_1\cap S_2$ or $p\in S_1\cap\{x_2^3-x_3^3=0\}$}. We again write $\Lambda_p=(f_p,g_p)_{\omega}$. A same calculation as in \cite[Prop.~(2.2),Lemma~(2.3)]{MR1480776} shows that each fiber over $\Spec(\mathscr{O}_{\mP^3_{\C}},p)$ has the same relative dimension and the closed fiber is a cone over a twisted cubic as described in Definition \ref{d:bundle}. Furthermore, in \cite[Lemma~2.4]{MR1480776}, Maeda shows the following facts:
\begin{enumerate}
    \item $V_{\Lambda_p}$ has an open affine cover 
$$V_{\Lambda_p}=U_1\cup U_2\cup U_3$$
where $U_1$ and $U_2$ are hypersurfaces in $\mA^3_{\mathscr{O}_{\mP^3_{\C},p}}$.

\item $U_3$ is a $(3,3)-$ complete intersection in $\mA^4_{\mathscr{O}_{\mP^3_{\C},p}}$.
\end{enumerate}
Notice that here $U_1,U_2,U_3$ do not have to be regular as $f_p,g_p$ are not part of a local parameters in the maximal ideal of the local ring at $p$, for some $p$. For example, when $p=[0:0:1:0]$. However, we can still conclude that $U_1,U_2,U_3$ are Cohen-Macaulay since they are complete intersection, hence so is $V_{\Lambda_p}$. So we have $\pi_p$ is flat by the "Miracle flatness" theorem.
\end{enumerate}
As the base field is $\C$, the argument above shows that the fiber over each geometric point is indeed one of the three cases in Definition \ref{d:bundle}. This shows that $V_{\Lambda}$ is a Brauer-Severi surface bundle over $\mP^3_{\C}$. We denote it by $Y$ in the following sections of this paper as before. 
\end{proof}

Now we explain which surfaces in $\mP^3_{\C}$ admit an associated Brauer-Severi surface bundle:
\begin{definition}
    \label{d: nontrivial triple cover \'etale in codimension 1}
Let $S$ be a reduced surface in $\mP^3_{\C}$ with irreducible components $S=S_1\cup S_2\cup \cdots \cup S_m$. Then we say $S$ admits a nontrivial triple cover \'etale in codimension 1 if there is nontrivial element in
$$\bigoplus_{i=1}^{n}H^1(\C(S_i),\Z/3)\cap \ker(\bigoplus_{C}(\partial^1_{C}))$$
Where $C$ runs over all irreducible curves in $\mP^3$, $\partial^1_{C}$ is the residue map as in Definition \ref{l:residue}.
\end{definition}
It is clear that any surface admits a nontrivial triple cover \'etale in codimension 1 will give us a 3-torsion Brauer class in $\C(\mP^3_{\C})$ by Bloch-Ogus sequence as discussed in Example \ref{e}. 

So with the proof of Theorem \ref{t:flatness}, we have:
\begin{cor}
    \label{c:associated Brauer-Severi surface bundle}
    Let $S\subset \mP^3_{\C}$ be a reduced surface which admits a nontrivial triple cover \'etale in codimension 1 (Definition \ref{d: nontrivial triple cover \'etale in codimension 1}). Assume the 3-torsion Brauer class given by the Bloch-Ogus sequence is represented by a cyclic algebra $\mathscr{A}$ of degree 3. Then there exists a Brauer-Severi surface bundle $Y_{S}\to \mP^3_{\C}$ with discriminant locus $S$ associated to $\mathscr{A}$. Furthermore, $Y_{S}$ is reduced. If the discriminant locus $S$ is good (Definition \ref{d:good}, indeed here we only need that part $(3)$ of this definition holds), then $Y_{S}$ is also irreducible, hence integral. 
\end{cor}
\label{c:integral total space}
\begin{proof}
The first part of this corollary directly follows from a similar discussion of local structures as in Theorem \ref{t:flatness}. Next, we show that $Y_{S}$ is reduced. Indeed, the map $Y_{S}\to \mP^3_{\C}$ is projective, hence closed. Then for any point $y\in Y_{S}$, there is a point $y'$ lying in a closed fiber such that $y$ specializes to $y'$. Since any localization of a reduced ring is again reduced, it suffices to check the local ring $\mathscr{O}_{Y_{S},y'}$ is reduced. Further more it suffices to assume $y'$ is a closed point. This can be directly checked using the explicit equations given in the proof of Theorem \ref{t:flatness}. (Details are discussed in Lemma \ref{l:reduceness of local model}.)

Now assume that the discriminant locus $S$ is good (Definition \ref{d:good}). Let $\pi$ denote the structure morphism $Y_{S}\to \mP^3_{\C}$. Consider the restricted Brauer-Severi surface bundle:
$$\pi^{'}:\pi^{-1}(\mP^3_{\C}-S)\to \mP^3_{\C}-S,$$ 
the base here is clearly irreducible. Since each fiber is a smooth Brauer-Severi surface, of the same dimension, and $\pi^{'}$ is projective, hence closed, we conclude that $\pi^{-1}(\mP^3_{\C}-S)$ is irreducible. Now if $Y_{S}$ is reducible, then by the argument above, $\pi^{-1}(S)$ is reducible. Hence there exist an irreducible component $S_i$ of $S$, such that $\pi^{-1}(S_i)$ is reducible (note that $Y_S$ is connected). But this is a contradiction to part $(3)$ of Definition \ref{d:good}.
\end{proof}

\section{The Specialization method and Desingularization}
\label{s:6}
In this section, we apply a specialization method introduced by Voisin in \cite{MR3359052}. It was further developed by Colliot-Th\'{e}l\`{e}ne and Pirutka in \cite{MR3481353} and modified by Schreieder in \cite[Proposition~26]{MR3909896}. We use this last version below as it is most suitabe to our example. We also refer to \cite[Section~2]{MR3849287} for a brief introduction of the Specialization method. The main difficulty of applying this to Example \ref{e} is that we need to construct an explicit desingularization 
$$f:\tilde{Y}\to Y$$
so that for all field extension $L/\C$, $f$ induces an isomorphism :
$$f_*:\CH_{0}(\tilde{Y}_L)\to \CH_{0}(Y_L)$$

Recently, Schreieder gave an alternate approach in a series of papers: \cite[Proposition~26]{MR3909896} and \cite{MR4013741}. Instead of constructing such a desingularization, Schreieder's result allows a purely cohomological criteria. Guided by his idea, we have the following known lemma(e.g. \cite[Proposition~4.8(a)]{schreieder2021unramifiedcohomologyalgebraiccycles}):

\begin{lemma}
\label{l:restriction}
Let $Y$ be a projective variety over a field $k$. Let $E\subset Y$ be an irreducible subvariety such that the local ring of $Y$ at the generic point $\eta_{E}$ of $E$, denoted by $\mathscr{O}_{Y,\eta_{E}}$, is a regular local ring.  Then there exists a restriction map:
$$\Res^{Y}_{E}:H^2_{nr}(k(Y)/k,\Z/3)\to H^2(k(E),\Z/3)$$
\end{lemma}
\begin{proof}
Let $\alpha\in H^2_{nr}(k(Y)/k,\Z/3)$ be an unramified Brauer class (Def \ref{d:unramified}).
Notice that by assumption, $\mathscr{O}_{Y,\eta_{E}}$ is a regular local ring with residue field $k(E)$ and fraction field $k(Y)$.
We have the following diagram:
      \[ 
  \begin{tikzcd}[row sep= 1 em, column sep= 1 em]
0 \arrow[d]\\
H^2(\mathscr{O}_{Y,\eta_{E}},\Z/3)\arrow[d]\arrow[r] & H^2(k(E),\Z/3)\\
H^2(k(Y),\Z/3) \arrow[d,"\oplus \partial^2_{\nu}"] \arrow[,ur,dotted]\\
\bigoplus H^1(k(\nu),\Z/3)\\
  \end{tikzcd}
  \]
here the left column is given by \cite[Theorem~3.8.3]{MR1327280}. The horizontal map is given by the functoriality in \'{e}tale cohomology. Now that $\alpha$ is an unramified Brauer class, it is killed by  $\oplus \partial^2_{\nu}$. Hence $\alpha$ comes from a class in $H^2(\mathscr{O}_{Y,\eta_{E}},\Z/3)$, which can be further mapped to $H^2(k(E),\Z/3)$ by the horizontal map. 
\end{proof}

We use notation in Example \ref{e} and Lemma \ref{l:restriction}. Let $U\subset Y$ be the smooth locus of $Y$. Let $\alpha_1\in Br(\C(\mP^3_{\C}))[3]$ be the Brauer class which has nontrivial residue $\gamma_1$ along $S_1$, and trivial residues everywhere else. By Lemma \ref{l:residue}, $\alpha_1$ can be represents by the cyclic algebra $$(\frac{x_2^3-x_3^3}{x_0^{3}},\frac{F_{S_1}}{x_0^9})_{\omega}.$$ By arguments in Example \ref{e}, $\alpha_1$ can be lifted to a nontrivial unramified Brauer class $$\tilde{\alpha}_1\in H^2_{nr}(\C(Y)/\C,\Z/3)$$

Now we prove that the second hypothesis in \cite[Proposition~26]{MR3909896} is true in our case:
\begin{lemma}
\label{l:restriction to E}
Let $\pi:Y\to \mP^3_{\C}$ be the Brauer-Severi surface bundle in Example \ref{e}. Let $U$ be the smooth locus of $Y$. Then there exists a resolution of singularities $f:\tilde{Y}\to Y$, such that for each irreducible component $E$ of $\tilde{Y}-f^{-1}(U)$, $\Res^{\tilde{Y}}_{E}(\tilde{\alpha}_1)$ is trivial.
\end{lemma}
\begin{proof}
The existence of resolution of singularities of $Y$ is guaranteed by Hironaka's theorem \cite{MR0199184}. We can further assume without loss of generality that each $E$ is a prime divisor of $\tilde{Y}$. 

Recall that $Y$ has singular locus of codimension at least 2. So for every irreducible component $E$ of $\tilde{Y}-f^{-1}(U)$, $f(E)$ has dimension at most $3$. On the other hand, since that the generic fiber of $\pi$ is smooth, and the generic fiber over each irreducible component of the discriminant locus (namely $S_1\cup S_2$) is the union of three standard Hirzebruch surfaces $\mathbb{F}_1$ described in Definition \ref{d:bundle}. We know $\pi(f(E))$ has dimension at most 1 (This follows from that local model of $\pi$ over $S_1$ and $S_2$ is smooth, see the discussion in Theorem \ref{t:flatness}). In other words, each $E$ in $\tilde{Y}$ would dominate a curve or a point in $S_1\cup S_2$. 

In the following of this proof, let $K=\C(\mP^3_{\C})$ be the function field of $\mP^3_{\C}$. Denote by $p_E$ the generic point of $\pi(f(E))$, by $K_{P_E}$ the field of fractions of the regular complete local ring $\hat{\mathscr{O}}_{\mP^3_{\C},p_E}$. We give a case by case argument according to the generic point $p_E\in \mP^3_{\C}$ of $\pi(f(E))$:


\begin{itemize}
    \item[\bf Case 1:]{\it $p_E\notin S_1$}. Then $\Res^{\tilde{Y}}_{E}(\tilde{\alpha}_1)=0$ simply follows from the fact that $p_E$ does not belongs to the discriminant locus of $\alpha_1$ (see e.g. the proof of \cite[Proposition~5.1(2)]{MR4013741}).
    \item[\bf Case 2:]{\it $p_E\in S_1-S_2$}. Consider the following commutative diagram coming from functoriality:
    \[ 
  \begin{tikzcd}[row sep= 1 em, column sep= 1 em]
H^2(K(Y_K),\Z/3\Z)\arrow[r] & H^2(K_{p_E}(Y_K),\Z/3\Z)\\
 H^2(K,\Z/3\Z) \arrow[u]\arrow[r] & H^2(K_{p_E},\Z/3\Z)\arrow[u]\\
  \end{tikzcd}
  \]
  Because $\tilde{\alpha}_1$ is unramified, by \cite[Proposition~2.5]{pirutka2023cubicsurfacebundlesbrauer}, it suffices to check $\alpha_1=0$ in 
  $H^2(K_{p_E}(Y_K),\Z/3\Z)$. Indeed, as $Y_K$ is the Brauer-Severi surface associate to the cyclic algebra $\mathscr{A}=(\frac{F_{S_2}(x_2^3-x_3^3)}{x_0^{21}},\frac{F_{S_1}}{x_0^9})_{\omega}$ and $F_{S_2}$ is a nonzero cube when $F_{S_1}=0$ in the residue field of $\mathscr{O}_{\mP^3_{\C},p_E}$. By Cohen's structure theorem \cite[Theorem~15]{5953cc44-5a9f-317f-90f3-83c005ff4b88}, the residue field embeds into $K_{P_E}$. Hence, after taking base changes to $K_{p_E}$,
  $$(\frac{F_{S_2}(x_2^3-x_3^3)}{x_0^{21}},\frac{F_{S_1}}{x_0^9})_{\omega}\cong (\frac{x_2^3-x_3^3}{x_0^{3}},\frac{F_{S_1}}{x_0^9})_{\omega}.$$
  This implies that $Y_K$ is also the Brauer-Severi surface associate to $\alpha_1$, we conclude that $\alpha_1=0$ in $H^2(K_{p_E}(Y_K),\Z/3\Z)$ by Amitsur's theorem \cite[Theorem~5.4.1]{MR2266528}.
    \item[\bf Case 3:]{\it $p_E\in S_1\cap S_2$ is a closed point, and $p_E\notin \{x_2^3-x_3^3=0\}$.} Then notice $\alpha_1$ can be represented by the cyclic algebra 
    $$(\frac{x_2^3-x_3^3}{x_0^{3}},\frac{F_{S_1}}{x_0^9})_{\omega}\cong (\frac{(x_2-x_3)^6}{(x_2^3-x_3^3)^2},\frac{F_{S_1}}{(x_2^3-x_3^3)^3})_{\omega}.$$
    Then $\frac{(x_2-x_3)^6}{(x_2^3-x_3^3)^2}$ is nonzero in the residue field of $\mathscr{O}_{\mP^3_{\C},p_E}$, which is $\C$. Hence $\frac{(x_2-x_3)^6}{(x_2^3-x_3^3)^2}$ is a cube in the residue field as $\C$ is algebraically closed. By Cohen's structure theorem, the residue field $\C$ embeds into $K_{P_E}$, hence $\frac{(x_2-x_3)^6}{(x_2^3-x_3^3)^2}$ is also a cube in $K_{P_E}$. This shows that $\alpha_1$ is trivial in $H^2(K_{P_E},\Z/3\Z)$. By the same commutative diagram as in case 2,
it is clear that $\alpha_1=0$ in  $H^2(K_{p_E}(Y_K),\Z/3\Z)$. We then have $\Res^{\tilde{Y}}_{E}(\tilde{\alpha}_1)=0$ by \cite[Proposition~2.5]{pirutka2023cubicsurfacebundlesbrauer}.
\item[\bf Case 4:]{\it $p_E$ is one of $[0:0:1:1]$,$[0:0:1:\omega]$,$[0:0:1:\omega^2]$.} In these cases, by the discussions in Case 4 of the proof of Theorem \ref{t:main}, we can choose another appropriate representing algebras of $\alpha_1$:
$$(\frac{x_2^6}{(x_1^3-x_2^3)(x_3^3-x_1^3)},\frac{F_{S_1}}{x_0^9})_{\omega}.$$
It is straight forward to check this is a representing algebra of $\alpha_1$ by applying Lemma \ref{l:residue}. And $\frac{x_2^6}{(x_1^3-x_2^3)(x_3^3-x_1^3)}$ is a nontrivial unit in the residue field of $\mathscr{O}_{\mP^3_{\C},p_E}$, which is $\C$. Hence the remaining proof can be done exactly same as in Case 3.
\item[\bf Case 5:]{\it $p_E=[0:1:0:0]$.} In this case, the proof are the same as in Case 4, by using the following representing algebra of $\alpha_1$:
$$(\frac{x_1^6}{(x_1^3-x_2^3)(x_3^3-x_1^3)},\frac{F_{S_1}}{x_0^9})_{\omega}.$$

\item[\bf Case 6:]{\it $p_E$ is one of the generic point of $D_1$, $D_2$ and $D_3$(Example \ref{e}).} 
Note that by the defining equations of $D_1,D_2$ and $D_3$, $\frac{x_2^3-x_3^3}{x_0^{3}}$ is always a nontrivial cube in the residue field of $\mathscr{O}_{\mP^3_{\C},p_E}$. Again as the residue field embeds into $K_{p_E}$, we get $\alpha_1=0$ in $H^2(K_{p_E}(Y_K),\Z/3\Z)$.
\end{itemize}
This completes the proof.
\end{proof}
 \section{Main result}
 \label{s:7}
 With Example \ref{e}, we prove Theorem \ref{t:final}:
 \primetheorem*
\begin{proof}
By  \cite[Proposition~26]{MR3909896} and Lemma \ref{l:restriction to E},we know the Brauer-Severi surface bundle constructed in Example \ref{e} can be used as a reference variety.

 To finish the proof, we need to construct a flat family of Brauer-Severi surface bundles over $\mP^3_{\C}$ with Example \ref{e} as one closed fiber with smooth general fiber. 
Start with the cyclic algebra from Example \ref{e}:
$$\mathscr{A}=(\frac{F_{S_2}(x_2^3-x_3^3)}{x_0^{21}},\frac{F_{S_1}}{x_0^9})_{\omega}$$
We consider two regular surfaces in $\mP^3_{\C}$:
$$G_1=\{x_0^9-x_1^9+x_2^8x_3+x_3^8x_2=0\}$$
$$G_2=\{x_0^{21}+x_1^{21}+x_2^{21}-x_3^{21}=0\}$$
By Lemma \ref{l:transversally}, both $G_1$ and $G_2$ are regular surfaces in $\mP^3_{\C}$, and they intersect transversally. Consider the following pencil of cyclic algebras:

$$\mathscr{A}_{[t_0:t_1]}=(\frac{t_0F_{S_2}(x_2^3-x_3^3)+t_1(G_2-F_{S_2}(x_2^3-x_3^3))}{x_0^{21}},\frac{t_0F_{S_1}+t_1(G_1-F_{S_1})}{x_0^9})_{\omega} $$
We denote 
$$t_0F_{S_2}(x_2^3-x_3^3)+t_1(G_2-F_{S_2}(x_2^3-x_3^3))$$ and 
$$t_0F_{S_1}+t_1(G_1-F_{S_1})$$ 
by $F_{S_2}^{[t_0:t_1]}$ and $F_{S_1}^{[t_0:t_1]}$ respectively. By Lemma \ref{l:singularties of F1G1} and Lemma \ref{l:singularties of F2G2}, when $[t_0:t_1]\neq [1:0]$, both
$F_{S_2}^{[t_0:t_1]}=0$ and $F_{S_1}^{[t_0:t_1]}=0$ are irreducible surfaces in $\mP^3$. Using Lemma \ref{l:residue}, the induced triple covers are given by 
$$\gamma_1^{[t_0:t_1]}=\frac{F_{S_2}^{[t_0:t_1]}}{x_0^{21}},\gamma_2^{[t_0:t_1]}=\frac{x_0^{9}}{F_{S_1}^{[t_0:t_1]}}$$
Similar to the discussion in Example \ref{e}, note that the residue of $\gamma_1^{[t_0:t_1]}$ of the valuation centered at the point $[0:\xi:1:1]$ is $1\in \Z/3$, where $\xi$ satisfies $\xi^9-2=0$. And the residue of $\gamma_2^{[t_0:t_1]}$ of the valuation centered at the point $[0:\psi:1:-1]$ is $2\in \Z/3$, where $\psi^{21}+2=0$. Hence both $\gamma_1^{[t_0:t_1]}$ and $\gamma_2^{[t_0:t_1]}$ are irreducible. By Corollary \ref{c:integral total space}, for any $[t_0:t_1]\in \mP^1_{\C}$, there exists an integral Brauer-Severi surface bundle $\mathscr{Y}_{[t_0:t_1]}\to \mP^3_{\C}$. 

By viewing $\mathscr{A}_{[t_0:t_1]}$ as a simple algebra over $\mP^1_{\C}\times \mP^3_{\C}$ and applying the construction of Theorem \ref{t:flatness} again, we have a Brauer-Severi surface bundle over $\mP^1_{\C}\times \mP^3_{\C}$ which can be viewed as a 1 dimensional family of of Brauer-Severi surface bundles over $\mP^3_{\C}$.
Let $\mathscr{Y}\to \mP^1_{\C}$ denote this family of Brauer-Severi surface bundles. We claim this is indeed a flat family. By \cite[Proposition~9.7]{MR0463157}, It is sufficient to check $\mathscr{Y}$ is integral. $\mathscr{Y}$ is irreducible because each closed fiber is an irreducible variety and the morphism $\mathscr{Y}\to \mP^1_{\C}$ is projective, hence closed. Now let $\hat{\mathscr{Y}}$ be the closed sub-scheme of $\mathscr{Y}$ with the same underlying topological space equipped with the reduced scheme structure, we have the following Cartesian diagram:

   \[ 
  \begin{tikzcd}[row sep= 1 em, column sep= 1 em]
\hat{\mathscr{Y}}_{[t_0:t_1]} \arrow[d,"i_{[t_0:t_1]}"]\arrow[r] & \hat{\mathscr{Y}} \arrow[d,"i"]\\
\mathscr{Y}_{[t_0:t_1]}\arrow[d]\arrow[r] &\mathscr{Y} \arrow[d]\\
\{[t_0:t_1]\} \arrow[r]& \mP^1_{\C}\\
  \end{tikzcd}
  \]

 On one hand, $i_{[t_0:t_1]}$ is a homeomorphism on topological spaces as the pullback of schemes by monomorphism coincide with topological pullback according to the explicit construction of fiber product of ringed spaces. On the other hand, $i_{[t_0:t_1]}$ is a closed immersion as a base change of the closed immersion $i$. Since $\mathscr{Y}_{[t_0:t_1]}$ is reduced as discussed above, we know $i_{[t_0:t_1]}$ is an isomorphism.

 Finally, by replacing $\mathscr{Y}$ by $\hat{\mathscr{Y}}$ if necessary, we get a 1 dimensional flat family of Brauer-Severi surface bundles over $\mP^3_{\C}$, with a special fiber $\mathscr{Y}_{[1:0]}$ (Example \ref{e} and Lemma \ref{l:restriction to E}) and a regular fiber $\mathscr{Y}_{[1:1]}$ (\cite[Theorem~2.1]{MR1480776}), hence we are done.

\end{proof}

\appendix
\section{calculations}
\label{s:A}
\begin{lemma}
\label{l:singulocus of P1P2}
Let $P_1=P_1(x_1,x_2,x_3)=(x_1^3-x_2^3)(x_2^3-x_3^3)(x_3^3-x_1^3)$, let $P_2=P_2(x_1,x_2,x_3)=x_1^3x_2^3x_3^3$. Then for any $[t_1:t_2]\neq [0:1]$ or $[1:0]$ or $[1:\sqrt{-27}]$ or $[1:-\sqrt{-27}]$, singular locus of the curve 
$t_1P_1+t_2P_2=0$ in $\mP^2_{\C}=\Proj \C[x_1,x_2,x_3]$ are precisely the three points:
$$[1:0:0],[0:1:0],[0:0:1]$$
\end{lemma}
\begin{proof}
Taking partial derivatives: 
$$t_1\frac{\partial P_1}{\partial x_1}+t_2\frac{\partial P_2}{\partial x_1}=0$$
$$t_1\frac{\partial P_1}{\partial x_2}+t_2\frac{\partial P_2}{\partial x_2}=0$$
$$t_1\frac{\partial P_1}{\partial x_3}+t_2\frac{\partial P_2}{\partial x_3}=0$$
Consider a point with $x_1=0$ lying in the curve. Then $P_2=0$, and this forces $P_1=x_2^3x_3^3(x_3^3-x_2^3)=0$ as $t_1\neq 0$. If one of $x_2$ or $x_3$ is 0, then we get the singularities listed in the statement of this lemma. If not, we have $x_2^3-x_3^3=0$, and the above partial derivatives can be simplified to obtain $x_2=0$, and hence $x_3=0$, which is a contradiction. Similar calculations work when $x_2=0$ or $x_3=0$. So all singularities when some $x_i$ equal to 0 are precisely the three points listed above.

In the following we assume all of $x_1$, $x_2$ and $x_3$ are nonzero. It is clear that for $i \neq j$, we have $x_i \frac{\partial P_2}{\partial x_i}=x_j\frac{\partial P_2}{\partial x_j}$. The above partial derivatives also tell us 
$$x_1 \frac{\partial P_1}{\partial x_1}=x_2\frac{\partial P_1}{\partial x_2}=x_3\frac{\partial P_1}{\partial x_3}.$$
By Euler's theorem on homogeneous polynomial, we know 
$$9P_1=x_1 \frac{\partial P_1}{\partial x_1}+x_2\frac{\partial P_1}{\partial x_2}+x_3\frac{\partial P_1}{\partial x_3}.$$
Hence we have 
$$x_1 \frac{\partial P_1}{\partial x_1}=3P_1,$$
which simplifies to 
$$(x_2^3-x_3^3)(x_2^3x_3^3-x_1^6)=0$$
Similarly,
$$(x_3^3-x_1^3)(x_1^3x_3^3-x_2^6)=0$$
$$(x_1^3-x_2^3)(x_1^3x_2^3-x_3^6)=0$$
By our assumption here, $P_2\neq 0$, so $P_1\neq 0$. Hence $x_i^3\neq x_j^3$, so we have 
$$x_2^3x_3^3-x_1^6=0$$
$$x_1^3x_3^3-x_2^6=0$$
$$x_1^3x_2^3-x_3^6=0$$
Hence the original partial derivatives simplify to 
$$3t_1(x_3^3-x_2^3)+t_2x_1^3=0$$
$$3t_1(x_2^3-x_1^3)+t_2x_3^3=0$$
$$3t_1(x_1^3-x_3^3)+t_2x_2^3=0$$
Viewing this as three linear equations of $x_i^3$, we can calculate the determinant of the coefficient matrix as $-t_2(27t_1^2+t_2^2)$. Hence when $[t_1:t_2]$ is not one of the four cases listed in the statement, the determinant is non-zero, we know $x_1^3=x_2^3=x_0^3=0$ is the only solution, which is impossible. 
\end{proof}
\begin{lemma}
\label{l:singularlocus of S2}
Using notations in Example \ref{e}, singular locus of $S_2$ consists of three lines:
$$L_1=\{x_0=0,x_1=0\}$$
$$L_2=\{x_0=0,x_2=0\}$$
$$L_3=\{x_0=0,x_3=0\}$$
Each $D_i$ exactly passes through 5 singular points of $S_2$, they are as follows:
$$D_1:[0:0:0:1], [0:0:1:0], [0:0:1:1], [0:0:1:\omega], [0:0:1:\omega^2]$$
$$D_2:[0:0:0:1], [0:1:0:0], [0:1:0:1], [0:1:0:\omega], [0:1:0:\omega^2]$$
$$D_3:[0:0:1:0], [0:1:0:0], [0:1:1:0], [0:1:\omega:0], [0:1:\omega^2:0]$$
\end{lemma}
\begin{proof}
Let $P_1=P_1(x_1,x_2,x_3)=(x_1^3-x_2^3)(x_2^3-x_3^3)(x_3^3-x_1^3)$, and let $P_2=P_2(x_1,x_2,x_3)=x_1^3x_2^3x_3^3$. Then any singular point of $S_2$ satisfies the one of the two systems of equations:
\begin{equation}
    \begin{cases}
      x_0=0\\
      P_2\frac{\partial P_1}{\partial x_1}+(P_1-2P_2)\frac{\partial P_2}{\partial x_1}=0\\
       P_2\frac{\partial P_1}{\partial x_2}+(P_1-2P_2)\frac{\partial P_2}{\partial x_2}=0\\
        P_2\frac{\partial P_1}{\partial x_3}+(P_1-2P_2)\frac{\partial P_2}{\partial x_3}=0\\
    \end{cases}\,.
\end{equation}
or 
\begin{equation}
    \begin{cases}
      x_0^9=-\frac{1}{2}(P_1-P_2)\\
      (P_1+P_2)\frac{\partial P_1}{\partial x_1}+(P_1-3P_2)\frac{\partial P_2}{\partial x_1}=0\\
       (P_1+P_2)\frac{\partial P_1}{\partial x_2}+(P_1-3P_2)\frac{\partial P_2}{\partial x_2}=0\\
        (P_1+P_2)\frac{\partial P_1}{\partial x_3}+(P_1-3P_2)\frac{\partial P_2}{\partial x_3}=0\\
    \end{cases}\,.
\end{equation}
In case $($A$.1)$, clearly points in $L_1\cup L_2\cup L_3$ are solutions of this system of equations. Assume $x_i\neq 0, i=1,2,3$. Then multiply the second equation of $($A$.1)$ by $x_1$, multiply the third equation of $($A$.1)$ by $x_2$, multiply the last equation of $($A$.1)$ by $x_3$ and add the resulting equations. By Euler's theorem on homogeneous polynomial, we have 
$$P_1-P_2=0$$
and the partial derivatives can be simplified as
$$x_0=0$$
$$\frac{\partial P_1}{\partial x_1}-\frac{\partial P_2}{\partial x_1}=0$$
$$\frac{\partial P_1}{\partial x_2}-\frac{\partial P_2}{\partial x_2}=0$$
$$\frac{\partial P_1}{\partial x_3}-\frac{\partial P_2}{\partial x_3}=0$$
Hence by Lemma \ref{l:singulocus of P1P2}, the set of all singularities in this case is identified to $L_1\cup L_2\cup L_3$.

In case $($A$.2)$, it is easy to check if some $x_i=0, i=1,2,3$, then either we are reduced to case $($A$.1)$ or we obtain that all of them are 0. So we again assume $x_i\neq 0, i=1,2,3$, and use the same trick as the previous case. By Euler's theorem on homogeneous polynomial, we have 
$$0=P_1^2+2P_1P_2-3P_2^2=(P_1-P_2)(P_1+3P_2)$$
If $P_1-P_2=0$, then $x_0=0$, we are reduced to the case $($A$.1)$. So the only new possibility is 
$P_1+3P_2=0$. Then the partial derivatives are exactly the partial derivatives for the curve $P_1+3P_2=0$. Again by Lemma \ref{l:singulocus of P1P2}, the set of singularities of $S_2$ are $L_1\cup L_2\cup L_3$.
\end{proof}

\begin{lemma}
\label{l:transversally}
Let
$$G_1=\{x_0^9-x_1^9+x_2^8x_3+x_3^8x_2=0\}$$
$$G_2=\{x_0^{21}+x_1^{21}+x_2^{21}-x_3^{21}=0\}$$
Then $G_1$ and $G_2$ are regular surface in $\mP^3_{\C}$. And furthermore $G_1$ and $G_2$ intersect transversally. 
\end{lemma}
\begin{proof}
$G_2$ is clearly a regular surface in $\mP^3_{\C}$. Taking partial derivatives of defining equation of $G_1$, the singular points are defined by the equations:
$$9x_0^8=0$$
$$-9x_1^8=0$$
$$8x_2^7x_3+x_3^8=0$$
$$8x_3^7x_2+x_2^8=0$$
This gives $x_0=x_1=0$. Note that if one of $x_2$ or $x_3$ is 0, so is the other. Assume $x_2\neq 0$ and $x_3\neq 0$, we get $8x_2^7+x_3^7=0, 8x_3^7+x_2^7=0$. Then again $x_2=x_3=0$, a contradiction. Hence $G_1$ is also a regular surface in $\mP^3_{\C}$.

To check $G_1$ intersects $G_2$ transversally, we prove by contradiction:
Assume there is a point $P=[x_0:x_1:x_2:x_3]\in G_1\cap G_2$, such that there exists a nonzero complex number $k$, with 
$$21x_0^{20}=k(9x_0^8)$$
$$-21x_1^{20}=k(9x_1^8)$$
$$21x_2^{20}=k(8x_2^7x_3+x_3^8)$$
$$-21x_3^{20}=k(8x_3^7x_2+x_2^8)$$
Where the left side of each equations is the partial derivatives of defining equations of $G_1$, and the right hand side is $k$ times the partial derivatives of defining equations of $G_2$.
We split into several cases:
\begin{enumerate}
    \item[\bf Case 1:] $x_2=0$ or $x_3=0$. In this case, we clearly have $x_2=x_3=0$, hence $x_0\neq 0$ and $x_1\neq 0$. So the partial derivatives with respect $x_0$ and $x_1$ tells us:
    $$x_0^{12}=\frac{9}{21}k=-x_1^{12}$$
    As $P\in G_1\cap G_2$, we also have 
    $$x_0^9-x_1^9=0$$
    $$x_0^{21}+x_1^{21}=0$$
    This forces $x_0=x_1=0$, which makes this case impossible.
    \item[\bf Case 2:]$x_2\neq 0$ and $x_3\neq 0$. Then the partial derivatives with respect to $x_2$ and $x_3$ shows that 
    $$\frac{x_2^{20}}{8x_2^7x_3+x_3^8}=\frac{k}{21}=-\frac{x_3^{20}}{8x_3^7x_2+x_2^8}$$
    (note that the denominator is nonzero, otherwise by the partial derivatives with respect to $x_2$ and $x_3$, one of $x_2$ or $x_3$ is 0. This contradicts to the assumption.)
    This can be simplified to 
   $$\frac{x_2^{21}}{8x_2^7+x_3^7}=-\frac{x_3^{21}}{8x_3^7+x_2^7}$$
   Since $x_3\neq 0$, we assume $x_3=1$ without lose of generality. Let $t=x_2^7$, we see the above relation shows that $t$ is a root of the following polynomial:
$$t^4+8t^3+8t+1=0$$
Now we have three subcases:
\begin{enumerate}
     \item[Sub-case 1:] $x_0=0$ and $x_1=0$. Then the defining polynomial of $G_2$ tells us 
     $$t^3-1=0$$
     This contradicts to the relation: $t^4+8t^3+8t+1=0$.
     \item[Sub-case 2:] $x_0\neq0$ and $x_1\neq 0$. In this case, a similar argument as in Case 1 shows that 
$$x_0^{12}+x_1^{12}=0$$
So we may assume $x_0=\tau x_1$, where $\tau$ satisfies $\tau^{12}+1=0$. Plug these information into defining equations of $G_1$ and $G_2$, we get: 
$$G_1:(\tau^9-1)x_1^9+x_2x_3(t+1)=0$$
$$G_2:-(\tau^9-1)x_1^{21}+t^3-1=0$$
Note that $\tau^9-1\neq 0$, and hence the above two equations shows that $t+1\neq 0$ and $t^3-1\neq 0$. 
Taking ratio of the above two equations, we have: 
$$-x_1^{12}=\frac{t^3-1}{x_2x_3(t+1)}$$
Compare the above relation with $G_1$, we have: 
$$x_1^3=\frac{(\tau^9-1)(t^3-1)}{x_2^2(t+1)^2}$$
Hence we get:
$$\frac{t^3-1}{\tau^9-1}=x_1^{21}=\frac{(\tau^9-1)^7(t^3-1)^7}{t^2(t+1)^{14}}$$
This is simplified to 
$$(\tau^9-1)^8(t^3-1)^6-t^2(t+1)^{14}=0$$
That is
$$(\tau^9-1)^8=\frac{t^2(t+1)^{14}}{(t^3-1)^6}$$
Since $t$ satisfies $t^4+8t^3+8t+1=0$, we list all roots of this polynomial:
$$t_1=-2-\frac{3}{\sqrt{2}}-\sqrt{\frac{15}{2}+6\sqrt{2}}$$
$$t_2=-2-\frac{3}{\sqrt{2}}+\sqrt{\frac{15}{2}+6\sqrt{2}}$$
$$t_3=-2+\frac{3}{\sqrt{2}}-\sqrt{-1}\sqrt{-\frac{15}{2}+6\sqrt{2}}$$
$$t_4=-2+\frac{3}{\sqrt{2}}+\sqrt{-1}\sqrt{-\frac{15}{2}+6\sqrt{2}}$$
By taking norm of both sides of $(\tau^9-1)^8=\frac{t^2(t+1)^{14}}{(t^3-1)^6}$ for each value of $t$ above, we see all four possible values of $t$ are impossible. (Indeed, one can check the norm of the left hand side has two possible estimated values: $0.118$ or $135.882$, while the norm of the right hand side has two possible estimated values: $0.00238$ or $14.2778$.)
\item[Sub-case 3:] One of $x_0$ or $x_1$ is $0$, and the other is nonzero. A similar discussion as in Sub-case 2 gives us 
$$1=\frac{t^2(t+1)^{14}}{(t^3-1)^6}$$
Again by taking norms of both sides, we see this is also impossible.
\end{enumerate}
 \end{enumerate}
 This completes the calculation.
\end{proof}
\begin{lemma}
\label{l:singularties of F1G1}
Use notations of Example \ref{e} and Lemma \ref{l:singularlocus of S2}, let $$G_1=\{x_0^9-x_1^9+x_2^8x_3+x_3^8x_2=0\}$$
Then for any $[t_0:t_1]\neq[1:0]\in \mP^1_{\C}$,
$$F_{S_1}^{[t_0:t_1]}=t_0F_{S_1}+t_1(G_1-F_{S_1})=0$$
defines irreducible surfaces in $\mP^3_{\C}$.

\end{lemma}

\begin{proof}
We can view $F_{S_1}^{[t_0:t_1]}$ as a polynomial in $x_0$:
$$F_{S_1}^{[t_0:t_1]}=t_0x_0^9+t_1(-x_1^9+x_2^8x_3+x_3^8x_2)+(t_0-t_1)P_1=0$$
Hence by the Eisenstein's criterion, it suffices to show the curve in $\mP^2_{\C}$ defined by the constant term 
$$t_1(-x_1^9+x_2^8x_3+x_3^8x_2)+(t_0-t_1)P_1=0$$
has a regular point. It is clear that $[0:1:0]$ is such a point.
\end{proof}
\begin{lemma}
\label{l:singularties of F2G2}
Use notations of Example \ref{e} and Lemma \ref{l:singularlocus of S2}, let 
$$G_2=\{x_0^{21}+x_1^{21}+x_2^{21}-x_3^{21}=0\}$$
Then for any $[t_0:t_1]\neq [1:0]\in \mP^1_{\C}$,
$$F_{S_2}^{[t_0:t_1]}=t_0F_{S_2}(x_2^3-x_3^3)+t_1(G_2-F_{S_2}(x_2^3-x_3^3))=0$$
define irreducible surfaces in $\mP^3_{\C}$.

\end{lemma}
\begin{proof}
View $F_{S_2}^{[t_0:t_1]}$ as a polynomial of $x_0$. As any factor of a homogeneous polynomial is also homogeneous, it suffices to show the constant term with respect to $x_0$ is itself irreducible. That is we need to show that

$$t_1x_1^{21}-(t_0-t_1)(x_2^3-x_3^3)P_2(P_1-P_2)+t_1(x_2^{21}-x_3^{21})=0$$
is irreducible. View the above polynomial as a polynomial of $x_1$, using Eisenstein criterion with the prime factor $(x_2-x_3)$, we get the conclusion.
\end{proof}
\begin{lemma}
\label{l:reduceness of local model} With notations as in the proof of Theorem \ref{t:flatness} and Corollary \ref{c:associated Brauer-Severi surface bundle}, we have that $Y_S$ is reduced.
\end{lemma}
\begin{proof}
As stated in the proof of Corollary \ref{c:associated Brauer-Severi surface bundle}, it suffices to check that over any closed point $p\in \mP^3_{\C}$, and any point $y$ lying in the fiber over $p$, the local ring 
$\mathscr{O}_{V_{\Lambda_p},y}\cong \mathscr{O}_{Y_{S},y}$ is reduced.

In the proof of Theorem \ref{t:flatness}, we provide open affine covers for each local model $V_{\Lambda_p}$. Hence it suffice to show the coordinate ring of each open affine set appearing in these open affine covers is reduced. We discuss the cases as in the proof of Theorem \ref{t:flatness} separately. First, Case 1 is trivial as the local model is regular.

In Case 2, we consider the the affine chart
$\{\xi_{11}=\xi_{21}=\xi_{31}=1\}$. Then its coordinate ring is 
$$R_p[\xi_{12},\xi_{13},\xi_{22},\xi_{23},\xi_{32},\xi_{33}]/(g_p\xi_{22}-\xi_{12},g_p\xi_{23}-\xi_{13},g_p\xi_{32}-g_p\xi_{12}, g_p\xi_{33}-\xi_{13})$$
It is easy to check that this defining ideal is radical. All the other affine charts can be checked similarly.

In Case 3, by \cite[Lemma~2.4]{MR1480776}, the local model has an open affine cover consisting of three open affine charts. The first two affine charts are both hypersurfaces in $\mA^3_{R_P}$  and since each is defined by an irreducible polynomial, each affine chart is reduced. For the third chart, we need to be careful since it is not a hypersurface. Its coordinate ring is given by 
$$R_p[x,y,z,w]/(F_1, F_2),$$
$$F_1=y^3-\omega x^2w+(1-\omega)xyz-f_p,$$
$$F_2=z^3-\omega^2 xw^2-(1-\omega)xyz-g_p.$$
for some $f_p.g_p\in R_p$. Here $\omega$ is a primitive third root of unity.
Recall that $R_p=\mathscr{O}_{\mP^3_{\C},p}=\C[\bar{x}_1,\bar{x}_2,\bar{x}_3]_{S_0}$, where $\bar{x}_i$ is a coordinate for some standard affine chart in $\mP^3_{\C}$, and $S_0$ is the multiplicative set $\C[\bar{x}_1,\bar{x}_2,\bar{x}_3]-m_p$ with $m_p$ the maximal ideal associates to $p$. Hence:
$$\displaystyle R_p[x,y,z,w]/(F_1, F_2) \cong(\C[\bar{x}_1,\bar{x}_2,\bar{x}_3,x,y,z,w]/(F_1, F_2))_{S_0}.$$
We check that this algebra is reduced using Serre's criterion. Namely, we verify whether our ring satisfies $(R_0)$ and $(S_1)$ \cite{stacks-project1}. Note that $F_1, F_2$ do not have common factors in the polynomial ring $\C[\bar{x}_1,\bar{x}_2,\bar{x}_3,x,y,z,w]$. Hence $F_1,F_2$ form a regular sequence, and so We have that $\C[\bar{x}_1,\bar{x}_2,\bar{x}_3,x,y,z,w]/(F_1, F_2)$ is a complete intersection. In particular, this affine chart is Cohen-Macaulay. This shows  $\C[\bar{x}_1,\bar{x}_2,\bar{x}_3,x,y,z,w]/(F_1, F_2))_{S_0}$ is also Cohen-Macaulay and hence satisfies Serre's condition $(S_1)$. On the other hand, one easily checks that $F_1$ and $F_2$  intersect transversally by showing that the rows of the jacobian matrix are never proportional along their intersection whenever both of them are nonzero. Note that this follows immediately since the part of the jacobian matrix corresponding to the variables $x, y, z,$ and $w$ already satisfies this property. Hence the $2\times 7$ Jacobian matrix of $F_1,F_2$ is not of full rank if and only if at least one of the two rows is zero. By the Jacobian criterion, these are precisely the singular points. We easily see that these points correspond to prime ideals in $\C[\bar{x}_1,\bar{x}_2,\bar{x}_3,x,y,z,w]$ containing one of the following ideals: $(x,y,f_p)$,$(z,x,y,g_p)$,$(z,x,w,g_p)$ or $(z,y,w,g_p)$. Hence singular set of $\C[\bar{x}_1,\bar{x}_2,\bar{x}_3,x,y,z,w]/(F_1, F_2)$ has codimension at least 3, which remains true by passing to the localization with respect to  $S_0$ as $f_p,g_p$ lives in the maximal ideal of $\mathscr{O}_{\mP^3_{\C},p}$. Thus $R_p[x,y,z,w]/(F_1, F_2)$ is regular in codimension 0, namely $(R_0)$. Hence this affine chart is also reduced. This completes the proof.
\end{proof}

\bibliographystyle{alpha}

\bibliography{ref.bib}

\end{document}